\documentclass[12pt,oneside]{amsart}
\usepackage{color,amssymb}

\newcommand{\R}{\mathbb{R}}

\newcommand{\ignore}[1]{}

\def \Aff{\hbox{\rm Aff}}
\def \Diff{\hbox{\rm Diff}}
\def \SL{\hbox{\rm SL}}

\def \R{{\mathbb R}}

\newtheorem{thm}{Theorem}[section]
\newtheorem{lemma}[thm]{Lemma}
\newtheorem{remark}[thm]{Remark}
\newtheorem{cor}[thm]{Corollary}
\newtheorem{prop}[thm]{Proposition}
\newtheorem{defin}[thm]{Definition}

\theoremstyle{plain}
\author{Alexander Gorodnik, Theron Hitchman, Ralf Spatzier}
\title[Regularity of conjugacies]{Regularity of conjugacies of algebraic actions of Zariski dense groups}
\address{School of Mathematics \\ University of Bristol \\ Bristol BS8 1TW, U.K.}
\email{a.gorodnik@bristol.ac.uk}
\address{Department of Mathematics \\ University of Northern Iowa \\Cedar Falls, IA 50614-0506}
\email{theron.hitchman@uni.edu}
\address{Department of Mathematics \\ University of Michigan \\ Ann Arbor, MI 48109-1043}
\email{spatzier@umich.edu}

\thanks{A.G. was supported by NSF grants DMS 0400631, 0654413 and RCUK Fellowship,
 R.S. was supported by NSF Grant DMS 0604857 }

\begin{document}
\begin{abstract}
Let $\alpha_0$ be an affine action of a discrete group $\Gamma$ on a compact homogeneous space $X$
and $\alpha_1$ a smooth action of $\Gamma$ on $X$ which is $C^1$-close to $\alpha_0$.
We show that under some conditions, every topological conjugacy between $\alpha_0$
and $\alpha_1$ is smooth. In particular, our results apply to Zariski dense subgroups
of $\hbox{SL}_d(\mathbb{Z})$ acting on the torus $\mathbb{T}^d$ and Zariski dense subgroups
of a simple noncompact Lie group $G$ acting on a compact homogeneous space $X$
of $G$ with an invariant measure.
\end{abstract}

\maketitle

{\small \tableofcontents}

\section{Introduction}

The investigation of rigidity properties has been at the forefront of research in dynamics in the past
two decades.  Of particular interest  has been the study of higher rank abelian groups and local rigidity
of their actions by  Hurder, Katok, Lewis, and the last author amongst others.  Remarkably,
many such actions cannot be perturbed at all, in the sense  that  any $C^1$-close perturbation is
$C^{\infty}$-conjugate to the original action. Critically, these groups contain higher rank abelian
groups. Similar results were found for higher rank semisimple Lie groups and their lattices by Hurder,
Lewis, Fisher, Margulis,  Qian and others.  We refer to \cite{Fisher-survey} for a more extensive survey
of these developments.

Smoothness of the conjugacy for these actions came as quite a surprise.  Classically, in fact, the stability results of Anosov and later Hirsch, Pugh and Shub  guaranteed a continuous conjugacy or orbit equivalence between a single Anosov  diffeomorphism or flow and their perturbations \cite{HPS}. Simple examples however show that such a conjugacy cannot be even $C^1$ in general.

In the present paper, we investigate similar regularity phenomena for affine actions of a large class of
groups.   Notably, our results do not require the presence of  higher rank subgroups
or any assumptions on the structure of the group.  In particular they
hold for discrete subgroups of rank one semisimple groups. We recall that a group acts {\em affinely} on
a homogeneous space $H/\Lambda$ for $H$ a Lie group and $\Lambda$ a discrete subgroup if every element
acts by an affine diffeomorphism i.e. one which lifts to a composite of a translation and an automorphism
on $H$.  We denote by $\Diff(X)$ the group of $C^\infty$-diffeomorphisms of a space $X$.

For simplicity let us mention two corollaries of our main theorem in Section \ref{s:main}.

\begin{thm}  \label{th:main1}
Let $\Gamma \subset \SL_d(\mathbb{Z})$ for $d \geq 2$  be a finitely generated Zariski
dense subgroup in $\SL_d(\R)$, and $\alpha_0$ the associated action on
the $d$-torus $\mathbb{T}^d$. If a perturbation  $\alpha_1:\Gamma\to\Diff(\mathbb{T}^d)$ is sufficiently $C^1$-close to
 $\alpha_0$, then any $C^0$-conjugacy $\Phi : \mathbb{T}^d \mapsto \mathbb{T}^d$ between $\alpha _0$ and $\alpha _1$
is a $C^\infty$-diffeomorphism.
\end{thm}

For $d=2$, E. Cawley   found a $C^{1+\alpha}$-regularity  result for Zariski-dense subgroups of $\SL_2(\mathbb{Z})$ acting on the 2-torus in  \cite{Cawley} in the early 1990's. Her techniques however are restricted to the 2-torus due to the use of $C^{1+\alpha}$-regularity of stable foliations.  Subsequently, the second author
obtained a general $C^{\infty}$-regularity  theorem for groups acting on general  tori in his thesis \cite{Hitchman}.

A second application of our main theorem to  actions on  homogeneous spaces of semisimple groups is novel.

\begin{thm} \label{th:main2}
Let $G$ be a connected simple noncompact Lie group, $\Lambda$ a cocompact lattice in $G$, and $\Gamma$ a 
finitely generated Zariski dense
subgroup of $G$.  Let $\alpha _0$ be the affine action of $\Gamma$ on $G/\Lambda$. If a
$C^\infty$-action $\alpha_1$ is  sufficiently $C^1$-close to
 $\alpha_0$, then any $C^0$-conjugacy $\Phi : G/\Lambda \mapsto G/\Lambda$ is a $C^\infty$-diffeomorphism.
 \end{thm}

Let us note that our techniques are based on certain mixing properties of
the actions and do not allow the treatment of actions on general nilmanifolds.

  Fisher and Hitchman recently proved a local rigidity theorem for   actions of  lattices with the
  Kazhdan property \cite{Fisher-Hitchman2006}.  We recall that an action $\alpha$ is called
  $C^{k,l}$-rigid if any $C^k$-close perturbation of the action is $C^l$-conjugate to $\alpha$.

 \begin{thm} [Fisher-Hitchman]
 Let $\Gamma$ be a lattice in a semisimple Lie group without compact factors  which satisfies Kazhdan's property.  Then any affine action $\alpha$ of $\Gamma$  is $C^{3,0}$-locally rigid.
 \end{thm}

 Fisher and Hitchman actually prove this for \emph{quasi-affine} actions, which are extensions of affine actions by isometries.
 Their technique is based on a type of heat flow.
If $\alpha$ does  not admit a common neutral direction, then Fisher and Hitchman's proof yields $C^{1,0}$-local rigidity.
Using our regularity result, we immediately obtain

 \begin{cor} %[Fisher-Hitchman]
  Let $G$ be a simple noncompact Lie group  which satisfies Kazhdan's property, $\Gamma$ a
  lattice in $G$, and $X$ a compact homogeneous space of $G$ supporting an invariant measure.  
Then the affine action of $\Gamma$ on $X$ is $C^{1, \infty}$-locally rigid.
\end{cor}

\begin{remark}{\rm
We can also deduce $C^{1,\infty}$-local rigidity for the action of a Kazhdan lattice $\Gamma$, embedded in
$\hbox{SL}_d(\mathbb{Z})$, on the torus $\mathbb{T}^d$ under the assumption that $\Gamma\times \Gamma$
is not contained in the subvarieties $\det([X^\ell,Y]-id)=0$, $\ell\ge 1$, $\phi(\ell)\le d^2$,
where $\phi$ is the Euler totient function (see Lemma \ref{th:bg}).
This assumption is needed to construct good pairs in $\Gamma$ (see Definition \ref{d:good}).
}
\end{remark}

Fisher and Hitchman proved $C^{\infty,\infty}$-local rigidity for a more general class of actions of cocompact lattices in the same groups \cite{Fisher-Hitchman2006}.   In particular their approach works on nilmanifolds.

At the heart of our argument lies the investigation of sequences of the form $\gamma ^{-n} \delta \gamma
^n$ for two hyperbolic elements $\gamma$ and $\delta$ in ``general position''.  Such elements always
exist in Zariski-dense groups.
The behavior of these sequences is badly divergent in directions transverse to the fast stable direction
of $\gamma$, and cannot be controlled.
However, these sequences do converge along the fast stable manifolds of $\gamma$.  This is elementary for
an affine action.  We prove $C^1$-convergence for the perturbed action.  These limiting maps along the
fast stable foliation of $\gamma$
form a rich system which acts transitively along the fast stable leaves under suitable conditions.
Moreover,  the conjugacy $\Phi$ between the actions will also intertwine these limiting maps along fast
stables.  It follows
that $\Phi$ has to be $C^1$ along each of these fast stable manifolds.  We prove smoothness in a separate
argument.

The proof of $C^1$-convergence is technically the most difficult piece of the argument. 
It requires careful estimates which are an adaptation of the proof of Livsic' theorem 
for cocycles with non-abelian targets.

The use of sequences of the form  $\gamma ^{-n} \delta \gamma ^n$ was introduced by Hitchman in his thesis
 \cite{Hitchman}.  His argument relied on the idea that the resulting limit maps 
along fast stable leaves often exhibit higher  rank abelian behavior which could then be used
to prove regularity similar to the case of actions by higher rank abelian groups.

  Let us comment that our arguments seem to be of rather general nature.  In the weakly hyperbolic
  setting,  the hard part  in proving local rigidity results lies in getting a $C^0$-conjugacy.
 Indeed, the common strategy for most of the known local rigidity results has been to show existence of a
 $C^0$-conjugacy and then improve the regularity.  Margulis--Qian in higher rank  and Fisher-Hitchman for
 all Kazhdan Lie groups have the most extensive   results \cite{MQ, Fisher-Hitchman}.
 The current paper shows regularity under rather general conditions, reducing smooth local rigidity to
 continuous local rigidity.  To pinpoint precisely when local rigidity holds appears difficult.  On the
 one hand, we have the results above for actions of lattices in the Kazhdan rank one groups.  On the
 other hand, Fisher found non-trivial affine deformations of actions  of lattices in $SO(n,1)$ resulting
 from ``bending lattices'' \cite{fisher-priv,Fisher-def}.  Finally, if the action has isometric
 directions, even  regularity becomes difficult as evidenced even in higher rank by the works of Fisher
 and Margulis \cite{Fisher-Margulis} and Fisher and Hitchman \cite{Fisher-Hitchman2006}.

\subsection{Acknowledgement}
We would like to thank  R.~Feres, A.~Gogolev, J.~Heinonen, B.~Kalinin, and B.~Schmidt for useful
discussions. We also would like to thank a referee for careful reading and
for pointing out some  deficiencies in our original proofs.
A.G. would like to express his thanks for hospitality to Princeton University, where 
part of this work was completed.

\section{Main result}\label{s:main}

Let $G$ be a connected Lie group, $\Lambda$ a cocompact lattice in $G$, and $X=G/\Lambda$.
The space $X$ is equipped with a finite invariant Radon measure.
The group $\Aff(X)$ of affine transformations of $X$ consists of maps of the form
$$
f:x\mapsto L_g\circ a(x),\quad x\in X,
$$
where $L_g$ denotes the left mutiplication action of $g\in G$ and $a$ is an automorphism of $G$ preserving $\Lambda$.
Every such map $f$ preserves the measure and
defines an automorphism $Df$ of $\hbox{Lie}(G)\simeq T_{e\Lambda}(X)$
given by
$$
Df:=\hbox{Ad}(g)\circ D(a)_e.
$$
We denote by $W^{min}_f$ \label{eq:wm}  the sum of the generalized eigenspaces of $Df$ with
eigenvalues of minimal modulus
and by $P^{min}_f:\hbox{Lie}(G)\to W^{min}_f$ the projection map along the other generalized eigenspaces. 
%Also, let $W^-_f$ be the sum of the root spaces of $Df$ with
%eigenvalues of modulus less than one.

\begin{defin}\label{d:good}
We call a pair $f,g\in \Aff(X)$ {\em good} if
the following conditions are satisfied:
\begin{enumerate}
\item[(i)] The map $Df$ is partially hyperbolic.
\item[(ii)] The map $Df :W^{min}_f\to W^{min}_f$ is semisimple.
\item[(iii)] The map $P^{min}_f \circ Dg:W^{min}_f\to W^{min}_f$ is nondegenerate.
\item[(iv)] For every subsequence $\{n_i\}$ , the sequence $\{f^{-n_i}gf^{n_i}(x)\}$ is
  dense in $X$ for $x$ in a set of full measure.
\end{enumerate}
If for $f\in \Aff(X)$, there exists $g\in \Aff(X)$ so that the pair $f,g$ is good, we say $f$ is a {\em good} mapping.
\end{defin}

\begin{remark} {\rm
In the case when the map $Df :W^{min}_f\to W^{min}_f$ does not have
a rotation component of infinite order (e.g., when $\dim W^{min}_f=1$), it suffices to
assume that the sequence $\{f^{-n}gf^{n}(x)\}$ is
  dense in $X$ for $x$ in a set of full measure.
In general, we have to pass to a subsequence to guarantee that
the maps $f^{-n}gf^n$ converge along the fast stable leaves as $n\to\infty$
(see Proposition \ref{p:alg}).
}
\end{remark}

The theorems stated in the introduction will be deduced from
the following general result:

\vspace{0.3cm}\noindent{\bf Main Theorem.} {\it
Let $\Gamma$ be a finitely generated discrete group and $\alpha_0:\Gamma\to\Aff(X)$ an affine action of $\Gamma$
such that 
\begin{itemize}
\item $(D\alpha_0)(\Gamma)$ acts irreducibly on $\hbox{\rm Lie}(G)$,
\item $\alpha_0(\Gamma)$ contains a good pair.
\end{itemize}
Let $\alpha_1:\Gamma\to\Diff(X)$ be a $C^\infty$-action of $\Gamma$ 
which is sufficiently $C^1$-close to $\alpha_0$.
Then every homeomorphism $\Phi:X\to X$ satisfying
$$
\Phi\circ \alpha_0(\gamma)=\alpha_1(\gamma)\circ \Phi\quad\hbox{for all $\gamma\in\Gamma$}
$$
is a $C^\infty$-diffeomorphism.
}
\vspace{0.3cm}

\begin{remark} \label{rem}{\rm
Irreducibility of the action of $\Gamma$ on $\hbox{Lie}(G)$ is used in the following places:
\begin{itemize}
\item In Section \ref{s1}, to deduce weak hyperbolicity (see (\ref{eq:wh})),
\item In Section \ref{s2}, to construct essential sets (see Lemma \ref{l:ess}),
\item In Section \ref{s5}, to deduce that $\Phi$ is $C^\infty$ from smoothness on subspaces of
the fast stable leaves
(see (\ref{eq:irrr})).
\end{itemize}
%It is clear from the argument in Section 3 that irreducibility can be replaced by the following condition:
%there exists good $f_0\in\alpha_0(\Gamma)$ such that
%$$
%\sum_{g\in\alpha_0(\Gamma)} (Dg)W^{min}_{f_0}=\hbox{\rm Lie}(G)\quad\hbox{and}\quad
%\bigcap_{g\in\alpha_0(\Gamma)} (Dg)W^{max}_{f_0}=0
%$$
%where $W^{max}_{f_0}$ is the sum of generalized eigenspaces complementary to $W^{min}_{f_0}$.
%Without the irreducibility assumption, one can still prove that
%the map $\Phi$ is $C^1$ restricted to the fast stable leaves of good $f_0\in\Aff(X)$
%(see Theorem \ref{th:ci}), and it is
%$C^\infty$ restricted to certain subspaces of the fast stable leaves (see Section \ref{s5}).
}\end{remark}

%\begin{remark} {\rm
%The proof of the Main Theorem also applies to perturbations with finite order of smoothness.
%One can show that for a $C^l$-action $\alpha_1:\Gamma\to\Diff^l(X)$, the conjugacy map $\Phi$ is
%$C^l$ on fast stable manifolds of good elements, and hence 
%by \cite[Theorem~3]{ll} or \cite[Theorem~1.1]{rt},
%$\Phi$ is in the Sobolev space of
%regularity $l$.  
%}\end{remark}

Existence of good pairs for some classes of affine actions will be proved in Section
\ref{sec:gen}. In particular, Theorem \ref{th:main1} follows from the Main Theorem
and Proposition \ref{p:torus}, and Theorem \ref{th:main2}  follows from the Main
Theorem and Proposition \ref{p:semi}.

\begin{proof}[Outline of the proof of the Main Theorem]
Irreducibility of $\Gamma$-action and property (i)  of a good pair are used to prove that $\Phi$ is bi-H\"older (Section \ref{s1}).
Next, irreducibility of the $\Gamma$-action and property (ii) of a good pair are used to show that 
$\Phi$ maps fast stable manifolds to fast stable manifolds (Section \ref{s2}).
Property (ii) is also used to show that a subsequence of maps $f^{-n}gf^n$
restricted to fast stable manifolds is precompact in the $C^0$-topology and, in fact, in the $C^1$-topology (Section \ref{s3}).
Then one utilizes property (iii)  of a good pair to deduce that the limits
of these maps are homeomorphisms and property (iv)  of a good pair to deduce that these limits generate transitive
$C^1$-action on fast stable manifolds. Using that $\Phi$ is a conjugacy between the constructed
$C^1$-actions, we show that $\Phi$ is $C^1$ along the fast stable leaves (Section \ref{s4}).
A more elaborate argument, which is based on the nonstationary Sternberg linearization
\cite{kl,GK,Guysinsky},  shows that $\Phi$ is $C^\infty$ along some subspaces of fast stable leaves.
Finally, we deduce that $\Phi$ is $C^\infty$ on $X$ from elliptic regularity
using  irreducibility of the $\Gamma$-action (Section \ref{s5}).
\end{proof}

\section{Proof of the Main Theorem}

We continue with the notation that $X=G/\Lambda$ is a compact quotient of a connected Lie group $G$ by a discrete subgroup $\Lambda \subset G$.

\subsection{$C^0$ implies H\"older}\label{s1}

In this section, we will prove that the conjugacy map $\Phi: X \to X$ in the Main Theorem
is bi-H\"{o}lder.  The proof is
similar to Proposition 5.7 of \cite{MQ}.  As they do not show that their map is H\"{o}lder, and also use
somewhat different
hypotheses, we will give a proof here for simplicity.

  Following \cite{MQ}, we say that a $C^1$-action $\alpha$ of a discrete group $\Gamma$ on a compact
  manifold $M$ is \emph{weakly hyperbolic} when there is a choice of finitely many elements
  $\gamma_1,\ldots,\gamma_k$ in $\Gamma$ such that each diffeomorphism $\alpha(\gamma_i)$ is partially
  hyperbolic and, for each point $x \in M$,
  \begin{equation}\label{eq:wh}
  \sum_{i=1}^{k} T_x W^s_{\alpha(\gamma_i)}(x) = T_x M,
  \end{equation}
  where $W^s_{\alpha(\gamma_i)}(x)$ denotes the stable manifold of $\alpha(\gamma_i)$ through $x$.

\begin{thm}\label{th:H}
Let $\Gamma$ be a finitely generated discrete group, $\alpha_0:\Gamma\to\Aff(X)$ be an affine 
weakly hyperbolic action, and $\alpha_1:\Gamma\to\Diff^1(X)$ 
a smooth action which is sufficiently $C^1$-close to $\alpha_0$.
Then every homeomorphism $\Phi:X\to X$ such that
$$
\Phi\circ \alpha_0(\gamma)=\alpha_1(\gamma)\circ \Phi\quad\hbox{for all $\gamma\in\Gamma$}
$$
is bi-H\"older.
 \end{thm}

The proof is divided into several lemmas.

\begin{lemma} \label{l:wh}
Let $f_1,\ldots,f_k$ be partially hyperbolic diffeomorphisms of $X$ such that
$$
\sum_{i=1}^{k} T_x W^s_{f_i}(x) = T_x M\quad\hbox{for all $x\in X$,}
$$
and $g_1,\ldots,g_k$ are $C^1$-close $C^1$-diffeomorphisms. Then
$g_i$'s are partially hyperbolic and
$$
\sum_{i=1}^{k} T_x W^s_{g_i}(x) = T_x M\quad\hbox{for all $x\in X$}.
$$
\end{lemma}

Lemma \ref{l:wh} follows from stability of partial hyperbolicity under perturbations
(see, for example, \cite[Lemma~3.5]{pe}).

\ignore{
\begin{proof}  Each $\alpha_0 (\gamma _i)$ acts normally hyperbolically w.r.t. the foliation $ N_i$.  By
  Theorem 6.1 and 6.8 of \cite{HPS}, any perturbation $\alpha _1 (\gamma _i)$ close enough in the
  $C^1$-topology is normally hyperbolic to a foliation $N _i '$.  Furthermore, the stable and unstable
  distributions  of $\alpha _1 (\gamma _i)$ are $C^0$-close to those of  $\alpha_0 (\gamma _i)$ It
  follows easily that $\alpha _1 (\Gamma)$ is still weakly  hyperbolic. Note that the elements
  $\gamma_1,\ldots,\gamma_k$
satisfy the definition of weak hyperbolicity for both actions.
   \end{proof}
}

\begin{lemma} \label{l:holder}
  Let $\Phi$ be a continuous conjugacy between two partially hyperbolic diffeomorphisms of a compact manifold. Then $\Phi$ is bi-H\"{o}lder continuous along the stable manifolds of these mappings.
\end{lemma}

Lemma \ref{l:holder} follows from the standard argument as in  \cite[Theorem 19.1.2]{Katok-Hasselblatt}.

\begin{lemma}\label{l:sch}
Let $\alpha:\Gamma\to\Diff^1(X)$ be a smooth weakly hyperbolic action
and $\gamma_1,\ldots,\gamma_k\in\Gamma$ satisfy (\ref{eq:wh}).
Then there exist $c,\epsilon>0$ such that for every $x,y\in X$ satisfying $d(x,y)<\epsilon$,
there exists a path $\ell$ from $x$ to $y$ which consists of $2k$ pieces contained in
stable manifolds of $\alpha(\gamma_1),\ldots,\alpha(\gamma_k)$,
and $L(\ell)\le c\, d(x,y)$.  
\end{lemma}

\begin{proof}
We will use an argument similar to \cite[Lemma~3.1]{sch}.

Let $d=\dim X$. There exists a family of (global) continuous unit vector fields
$v_1,\ldots,v_d$ that span the tangent space at every point and
for some $1=d_0\le d_1\le\cdots\le d_{k}=d+1$ and every $i=1,\ldots, k$,
the vectors $v_{d_{i-1}},\ldots,v_{d_{i}-1}$ are contained in
the stable distribution of $\alpha(\gamma_i)$.
Let $\delta>0$. There exists $\delta'>0$ such that $d(u,w)<\delta'$ implies that
$d(v_i(u),v_i(w))<\delta$ for all $i$.
By \cite[Corollary~4.5]{kp}, for every $x\in X$, there exists $\epsilon(x)>0$ such that
every $y\in B_{\epsilon(x)}(x)$ can be connected to $x$ by a path $\ell$ of length
at most $\delta'/2$, and for some 
$0=t_0\le t_1\le\cdots\le t_{d}=L(\ell)$, we have $\ell'(t)=v_i(\ell(t))$ when $t\in [t_{i-1},t_i)$.
Let $\epsilon>0$ be the Lebesgue number of the cover $\{B_{\epsilon(x)}(x)\}$.
Then every $y_1,y_2\in X$ such that $d(y_1,y_2)<\epsilon$ are connected by a path $\ell$
which consists of $2k$ pieces tangent to $v_j$'s and $L(\ell)<\delta'$.
To estimate the distance $d(y_1,y_2)$, we may assume, without loss of generality,  that we work in an open neighborhood of $\mathbb{R}^d$ 
equipped with the standard metric. 
By the triangle inequality,
\begin{align*}
\|y_1-y_2\|&=\left\|\sum_{i} \int_{t_{i-1}}^{t_i}v_{i}(\ell(t))dt\right\|\\
&\ge \left\|\sum_{i} (t_i-t_{i-1})v_{i}(y_1)\right\|-\sum_{i} \int_{t_{i-1}}^{t_i} \|v_{i}(\ell(t))-v_{i}(y_1)\|dt\\
&\ge (c-\delta) L(\ell)
\end{align*}
where
$$
c=\min\left\{\left\| \sum_i s_i v_i(y)\right\|:\, \sum_i |s_i|= 1,\, y\in X  \right\}>0.
$$
Taking $\delta$ sufficiently small, this implies the
estimate for $L(\ell)$. Since the stable distributions are uniquely integrable, $\ell([t_{i-1},t_i))$
is contained in the stable manifold of $\alpha(\gamma_{i})$.
\end{proof}

\vspace{0.4cm}
\begin{proof}[Proof of Theorem \ref{th:H}]
Let $\gamma _1, \ldots, \gamma _k\in\Gamma$ be elements satisfying (\ref{eq:wh}).
By Lemma \ref{l:holder}, the map $\Phi$ is bi-H\"older restricted to the stable
manifolds of $\alpha_0(\gamma_i)$'s. By Lemma \ref{l:sch}, for sufficiently close
$x,y\in X$, there exist points $x_0=x,x_1,\ldots,x_{2k}=y$ such that $x_{j-1}$ and $x_{j}$
are on the same stable manifold of some $\alpha_0(\gamma_{i_j})$,
and $d(x_{j-1},x_{j})\le c\, d(x,y)$. Then
\begin{align*}
d(\Phi(x),\Phi(y))&\le \sum_{j=1}^{2k} d(\Phi(x_{j-1}),\Phi(x_j)) \le \sum_{j=1}^{2k} c_j
d(x_{j-1},x_j)^{\theta_j}\\
&\le \left(\sum_{j=1}^{2k} c_jc^{\theta_j}\right)d(x,y)^\theta
\end{align*}
where $\theta=\min \theta_j$. 

By Lemma \ref{l:wh}, the action $\alpha_1$ is also weakly hyperbolic. 
Then the proof that $\Phi^{-1}$ is H\"older follows the same argument.
\end{proof}

\ignore{
\begin{lemma} Let  $\Phi$ be a continuous conjugacy between a weakly hyperbolic affine algebraic action $\alpha _0$ and a $C^1$-close perturbation $\alpha _1$.   Then $\Phi$  is H\"{o}lder.\end{lemma}

\begin{proof}   Since $\alpha _1 $ is $C^1$-close to $\alpha _0$, we may assume that there is a finite collection of elements $\gamma _1, \ldots, \gamma _k$ from $\Gamma$ which act as partially hyperbolic diffeomorphisms  under both
$\alpha _0$ and $\alpha _1$, and that the stable distributions for $\alpha _0 (\gamma _1), \ldots, \alpha _0 (\gamma _k)$ span the tangent space of $M$.  As the stable manifolds of $\alpha _0 (\gamma _1), \ldots, \alpha _0 (\gamma _k)$ are smooth foliations of $M$, given two points $p$ and $q$ in $M$, we can easily find  a path $c$ from  $p$ to $q$  composed  of pieces which lie entirely in the stable manifolds of  $\alpha _0 (\gamma _1), \ldots, \alpha _0 (\gamma _k)$, and such that the length $l(c)$ of $c$
satisfies  $l(c) < A \: d(p,q) $ for some constant $A>0$ independent of $p$ and $q$.
The previous lemma now immediately implies that $\Phi$ is H\"{o}lder.
\end{proof}
}

\subsection{Invariance of fast stable manifolds}\label{s2}

Let $f\in\Diff(X)$, and the tangent bundle $TX$ has continuous $f$-invariant splitting
\begin{equation}\label{eq:hyp}
TX=E^-\oplus E^+
\end{equation}
such that for some $\lambda\in (0,1)$ and $\mu>\lambda$,\footnote{The notation $A\ll  B$ means that there exists $c>0$, independent
of other parameters, such that $A\le c\, B$.}
\begin{align}\label{eq:hyp2}
\|D(f^n)_xv\|&\ll \lambda^n\|v\|\quad\hbox{for all $n\ge 0$, $x\in X$, and $v\in E_x^-$},\\
\|D(f^n)_xv\|&\gg \mu^n\|v\|\quad\hbox{for all $n\ge 0$, $x\in X$, and $v\in E_x^+$}.\nonumber
\end{align}
We recall (see, for example, \cite[Theorem~4.1]{pe}) that the distribution $E^-$ is integrable to 
the {\it fast stable} foliation $\{W^{fs}_{f}(x)\}_{x\in X}$, and
this foliation is H\"older continuous with $C^\infty$-leaves.
%Note that $\{W^{fs}_{f}(x)\}_{x\in X}$ is a foliation with $C^\infty$-leaves which depend continuously
%in $C^\infty$-topology on the base point. 
We denote by $d^{fs}$ the induced metrics on the leaves of this foliation.
For $\rho>\lambda$ and $x,y\in X$ such that $y\in W^{fs}_f(x)$,
$$
d^{fs}(f^n(x),f^n(y))\ll \rho^n d^{fs}(x,y).
$$
There exists $\epsilon_0>0$ such that for every 
$z,w\in X$ satisfying $w\in W^{fs}_{f}(z)$ and $d^{fs}(z,w)<\epsilon_0$, we have
\begin{equation}\label{eq:fs_s}
d^{fs}(z,w)\ll d(z,w)\le d^{fs}(z,w).
\end{equation}

Let $f_0\in \Aff(X)$ be such that $Df_0$ is partially hyperbolic,
and $\lambda_0<\mu_0$ denote the least two absolute values of the eigenvalues of $Df_0$.
If $f\in\Diff(X)$ is a $C^1$-small perturbation of $f_0$, then we have a splitting as above
with $\lambda=\lambda_0+\epsilon$ and $\mu=\mu_0-\epsilon$ for some small $\epsilon>0$,
depending on $d_{C^1}(f,f_0)$
(see \cite[Lemma~3.5]{pe}). The fast stable manifolds $W^{fs}_{f}(x)$ are defined with respect to this splitting.
Note that
$$
W^{fs}_{f_0}(x)=\exp(W^{min}_{f_0})x
$$
where $\exp$ is the Lie exponential map, and $W^{min}_{f_0}$ is defined as on page \pageref{eq:wm}.

The aim of this section is to prove the following theorem.

\begin{thm}\label{th:ffs}
Let $\alpha_0:\Gamma\to\Aff(X)$ and $\alpha_1:\Gamma\to\Diff(X)$
be $C^1$-close actions of a finitely generated discrete group $\Gamma$, and let $\Phi:X\to X$ be a homeomorphism such that
$$
\Phi\circ \alpha_0(\gamma)=\alpha_1(\gamma)\circ \Phi\quad\hbox{for all $\gamma\in\Gamma$}.
$$
Assume that $(D\alpha_0)(\Gamma)$ acts irreducibly on $\hbox{\rm Lie}(G)$.
Then for every partially hyperbolic $f_0:=\alpha_0(\gamma)$ and $f:=\alpha_1(\gamma)$, $\gamma\in\Gamma$,
such that $Df_0$ is semisimple on $W^{min}_{f_0}$,
$$
\Phi(W^{fs}_{f_0}(z))= W^{fs}_{f}(\Phi(z)) \quad \hbox{for all $z\in X$.}
$$
Moreover, the map $\Phi$ is bi-H\"older with respect to the induced metrics on
fast stable leaves of $f_0$ and $f$.
\end{thm}

Let us start with some preliminary reductions.
%Let $f_0=\alpha_0(\gamma)$ and $f=\alpha_1(\gamma)$.
We will prove that 
\begin{equation}\label{eq:sub}
\Phi^{-1}(W^{fs}_f(z))\subset W^{fs}_{f_0}(\Phi^{-1}(z))\quad\hbox{for all $z\in X$.}
\end{equation}
This also implies that the equality.
Indeed, it follows from \eqref{eq:sub} that every leaf $W^{fs}_{f_0}(\Phi^{-1}(z))$ is
a disjoint union of sets of the form $\Phi^{-1}(W^{fs}_{f}(y))$ for some $y\in X$.
By \cite[Lemma~3.5]{pe}, the fast stable leaves of $f_0$ and $f$ have the same dimension. 
Hence, by the invariance of domain, every set $\Phi^{-1}(W^{fs}_{f}(y))$ is open in $W^{fs}_{f_0}(\Phi^{-1}(z))$.
Since $W^{fs}_{f_0}(\Phi^{-1}(z))$ is connected, we deduce that
$$
\Phi^{-1}(W^{fs}_{f}(z))= W^{fs}_{f_0}(\Phi^{-1}(z)).
$$

%We denote by $d^{fs}$ the induced Riemannian metric on the fast stable leaves $W^{fs}_{f}(z)$.
Let 
\begin{equation}\label{eq:se}
\mathcal{S}_{\epsilon'}(x)=\{\Phi^{-1}(z):\; z\in W^{fs}_{f}(\Phi(x)),\, d^{fs}(z,\Phi(x))<\epsilon'\}.
\end{equation}
We will show that there exists $\epsilon'\in (0,\epsilon_0)$ such that
for every $x\in X$,
$$
\mathcal{S}_{\epsilon'}(x)\subset W^{fs}_{f_0}(x).
$$
This will imply the theorem.

First, we observe the following property of points lying on the same fast stable leaf for affine actions:

\begin{prop}\label{p:fs0}
Let $f_0,g_0\in\Aff(X)$ be such that $(Df_0)|_{W^{min}_{f_0}}$ is semisimple.
Then there exists $c>0$ such that for every $z,w\in X$ satisfying $w\in W^{fs}_{f_0}(z)$ and $n\ge k\ge 0$,
$$
d(f_0^{-k}g_0f_0^n(z),f_0^{-k}g_0f_0^n(w))\le c\,\lambda_0^{n-k} d^{fs}(z,w)
$$
where $\lambda_0$ is the least absolute value of the eigenvalues of $Df_0$.
\end{prop}

\begin{proof}
It suffices to prove the proposition when $d^{fs}(z,w)$ small.
Write $w=\exp(v)z$ for $v\in W^{min}_{f_0}$. 
Then 
$$
w=\exp(D(f_0^{-k}g_0f_0^n)v)f_0^{-k}g_0f_0^n(z),
$$
and
it suffices to show that for a norm on $\hbox{Lie}(G)$,
$$
\|D(f_0^{-k}g_0f_0^n)v\|\ll \lambda_0^{n-k}\|v\|,
$$
which is easy to check.
\end{proof}

A similar but weaker property also holds for small perturbations of affine actions:

\begin{prop}\label{p:fs}
Let $f_0\in \Aff(X)$, $g\in \Diff(X)$, and $\nu>1$. Then  
there exists $c>0$ such that for any sufficiently $C^1$-small
perturbations $f\in \Diff(X)$ of $f_0$, $z,w\in X$ satisfying $w\in W^{fs}_f(z)$, and $n\ge 0$,
\begin{equation}\label{eq:ws}
d(f^{-n}gf^n(z),f^{-n}gf^n(w))\le c\,\nu^n d^{fs}(z,w).
\end{equation}
\end{prop}

\begin{proof}
Let $\lambda_0$ denote the least absolute value of the eigenvalues of $Df_0$.
Take $\lambda_-<\lambda_0<\lambda_+$ such that $\frac{\lambda_+}{\lambda_-}<\nu$.
For $f$ sufficiently $C^1$-close to $f_0$, we have
$$
\|D(f^{-n})_u\|\ll \lambda_-^{-n}\quad\hbox{for all $u\in X$ and $n\ge 0$,}
$$
and
$$
\left\|D(f^{n})|_{T_u(W^{fs}_f(u))}\right\|\ll \lambda_+^{n}\quad\hbox{for all $u\in X$ and $n\ge 0$.}
$$ 
This implies that
$$
\left\|D(f^{-n}gf^n)|_{T_u(W^{fs}_f(u))}\right\|\ll \left(\frac{\lambda_+}{\lambda_-}\right)^{n}\quad\hbox{for all $u\in X$ and $n\ge 0$.}
$$ 
Let $\ell$ be a smooth curve in $W^{fs}_f(z)$ from $z$ to $w$ such that $L(\ell)=d^{fs}(z,w)$.
Then
$$
L(f^{-n}gf^n(\ell))\ll\left(\frac{\lambda_+}{\lambda_-}\right)^{n}L(\ell)<\nu^n d^{fs}(z,w)
$$
for all $n\ge 0$. This proves the proposition.
\end{proof}

It turns out that property (\ref{eq:ws}) characterizes points lying on the same
fast stable leaves.
This observation is crucial for the proof of Theorem \ref{th:ffs}
and is the main point of Theorem \ref{th:quant} below.
Since the proof of Theorem \ref{th:quant} is quite involved, we first present
its linear analogue -- Proposition \ref{p:lin}. Although the argument in the proof of Theorem \ref{th:quant}
follows the same idea, it requires more delicate quantitative estimates because
we have to work in injectivity neighborhoods of the exponential map.

Let $A\in\hbox{GL}_l(\mathbb{R})$. We denote by
$\lambda_1<\cdots<\lambda_d$ the absolute values of the eigenvalues of $A$,
and $P_i$ denote the projection to the
sum of the generalized eigenspaces of $A$ corresponding to $\lambda_i$
along the other eigenspaces. 

\begin{prop}\label{p:lin}
Let $B_1,\ldots, B_k\in \hbox{\rm GL}_l(\mathbb{R})$ be such that
for some $\eta>0$
\begin{equation}\label{eq:alpha}
\max_k\|P_1B_kv\|>\eta\|v\|\quad \hbox{for all $v\in\mathbb{R}^l$.}
\end{equation}
Then there exists $\nu>1$ such that
$$
W^{min}_A=\{v:\, \max_k \|A^{-n}B_kA^n\,v\|=O(\nu^n)\quad \hbox{as $n\to\infty$}\}.
$$
\end{prop}

\begin{proof}
For every small $\rho>0$ there exists a norm on $\mathbb{R}^l$ (see \cite[Proposition~1.2.2]{Katok-Hasselblatt})
such that $\|v_1+v_2\|=\|v_1\|+\|v_2\|$ for $v_1$ and $v_2$ in different generalized eigenspaces and
\begin{align*}
(\lambda_i-\rho)\|v\| &\le \|Av\|\le (\lambda_i+\rho)\|v\|,\quad v\in \hbox{im}(P_i).
\end{align*}
The parameter $\rho$ is fixed, but has to be chosen sufficiently small so that
$$
(\lambda_1-\rho)^{-1}(\lambda_1+\rho)<\min_{i>1} (\lambda_1+\rho)^{-1}(\lambda_i-\rho).
$$
It follows from (\ref{eq:alpha}) that
\begin{align*}
\max_k \|A^{-n}B_kA^nv\| &\ge \max_k (\lambda_1+\rho)^{-n}\|P_1 B_kA^nv\|\\
&\ge (\lambda_1+\rho)^{-n}\eta \|A^nv\|\\
&\ge (\lambda_1+\rho)^{-n}\eta \sum_i  (\lambda_i-\rho)^n \|P_iv\|.
\end{align*}
We take $\nu>1$ such that 
$$
\hbox{$\nu<(\lambda_1+\rho)^{-1}(\lambda_i-\rho)$ for $i>1$ and
  $\nu>(\lambda_1-\rho)^{-1}(\lambda_1+\rho)$.}
$$
Then $\max_k \|A^{-n}B_kA^nv\|=O(\nu^n)$ implies that $P_iv=0$ for $i>1$.
Also, for $v\in W^{min}_A$,
\begin{align*}
\max_k \|A^{-n}B_kA^nv\| &\le   (\lambda_1-\rho)^{-n}\left(\max_k \|B_k\|\right ) (\lambda_1+\rho)^n=O(\nu^n).
\end{align*}
This proves the proposition.
\end{proof}

Proposition \ref{p:fs} and (\ref{eq:fs_s}) imply that uniformly on $z,w\in X$, satisfying
$w\in W^{fs}_{f}(z)$ and $d^{fs}(z,w)<\epsilon_0$, and $n\ge 0$,  we have
$$
d(f^{-n}gf^n(z),f^{-n}gf^n(w))\ll \nu^n d(z,w).
$$
Now we take $g=\alpha_1(\delta)$ and $g_0=\alpha_0(\delta)$ for some $\delta\in\Gamma$.
Since the action of $\Gamma$ on $\hbox{Lie}(G)$ is irreducible, $\alpha_0$ is weakly hyperbolic.
Hence, by Theorem \ref{th:H}, the conjugacy map $\Phi$ and its inverse are H\"older with some exponent
$\theta>0$. It follows that
uniformly on $x,y\in X$, satisfying $y\in \mathcal{S}_{\epsilon'}(x)$, and $n\ge 0$,
\begin{align}\label{eq:ff_0}
d(f_0^{-n}g_0f_0^n(x),f_0^{-n}g_0f_0^n(y))
&\ll d(f^{-n}gf^n(\Phi(x)),f^{-n}gf^n(\Phi(y)))^\theta\\
&\ll \nu^{\theta n} d(\Phi(x),\Phi(y))^\theta\nonumber\\
&\ll \nu^{\theta n} d(x,y)^{\theta^2}.\nonumber
\end{align}

Let $\lambda_1<\cdots<\lambda_d$ be the absolute values of the eigenvalues of $Df_0$
and $P_i$ denote the projection from $\hbox{Lie}(G)$ to the
sum of the generalized eigenspaces of $Df_0$ corresponding to $\lambda_i$
along the other generalized eigenspaces. 

We say that a set $\{g_1,\ldots, g_l\}\subset \Aff(X)$ is {\it essential} for $f_0$
if  for some $\eta>0$ and  every $v\in \hbox{Lie}(G)$,
\begin{align*}
\max_k\|P_1(Dg_k)v\|>\eta\|v\|.
\end{align*}
Note this definition does not depend on a choice of the norm.
Existence of essential sets follows from the following lemma:

\begin{lemma}\label{l:ess}
A set $g_1,\ldots,g_l \in \Aff(X)$ is essential if and only if
\begin{equation}\label{eq:star}
\bigcap_{k=1}^l (Dg_k)^{-1}\hbox{\rm ker}(P_1)=0.
\end{equation}
In particular, every subgroup $\Gamma\subset\Aff(X)$ such that $D\Gamma$ 
acts irreducibly on $\hbox{\rm Lie}(G)$ contains an essential set.
\end{lemma}

Although the group $\Gamma$ in the Main Theorem needs to be finitely generated, this assumption is not
needed in Lemma \ref{l:ess}.

\begin{proof}
Since the map 
$$
v\mapsto (P_1(Dg_k)v:\,k=1,\ldots l):\hbox{Lie}(G)\to\hbox{Lie}(G)^l
$$
is injective when (\ref{eq:star}) holds, one can take
$$
\eta=\min\{\max_k\|P_1(Dg_k)v\|: \|v\|=1\}>0.
$$
The converse is also clear.

To prove the second claim, we observe that there exists a subset 
$\{g_1,\ldots, g_l\}\subset \Gamma$ such that
$$
\bigcap_{k=1}^l (Dg_k)^{-1}\hbox{ker}(P_1)=\bigcap_{g\in\Gamma} (Dg)^{-1}\hbox{ker}(P_1),
$$
and this space is zero by irreducibility.
\end{proof}

The following theorem is the main ingredient of the proof of Theorem \ref{th:ffs}:

\begin{thm}\label{th:quant}
There exists $\nu=\nu(\vartheta,f_0)>1$ such that
given constants $a,\vartheta>0$, a map $f_0\in\Aff(X)$ such that $Df_0$ is semisimple on $W^{min}_{f_0}$,
an essential set $g_1,\ldots,g_l\in\alpha_0(\Gamma)$,
and a family of subsets $\mathcal{L}_\epsilon(x)$, $x\in X$, of $X$ that satisfy
\begin{enumerate}
\item[(i)] $x\in \mathcal{L}_\epsilon(x)\subset B_\epsilon(x)$,
\item[(ii)] $f_0^{-1}(\mathcal{L}_\epsilon(x))\supset \mathcal{L}_\epsilon(f_0^{-1}(x))$,
\item[(iii)] for every $y\in\mathcal{L}_\epsilon(x)$ and $n\ge 0$,
\begin{equation}\label{eq:cb}
\max_k d(f_0^{-n}g_kf_0^{n}(x),f_0^{-n}g_kf_0^{n}(y))\le a\nu^n d(x,y)^\vartheta,
\end{equation} 
\end{enumerate}
one can choose $\epsilon>0$ such that
$$
\mathcal{L}_\epsilon(x)\subset  W^{fs}_{f_0}(x)\quad\hbox{for every $x\in X$}.
$$
\end{thm}

\begin{proof}[Outline of the proof of Theorem \ref{th:quant}] We first observe that
the sets $\mathcal{L}_\epsilon(x)$ lie in ``cones'' around $W^{fs}_{f_0}(x)$
where the size of the cones is controlled by $\nu$
and can be made sufficiently small (Lemma \ref{l:cone1}).
Note that this argument is analogous to the proof of Proposition
\ref{p:lin}, but we can only derive a weaker conclusion because
one has to work in injectivity neighborhoods of the exponential map. 
In the next step, we show that applying the map $f_0^{-1}$, the size of the cones can be made arbitrary small
(Lemma \ref{l:cone2}). This implies the theorem.
\end{proof}

We fix a norm on $\hbox{Lie}(G)$, depending on parameter $\rho>0$, as in the proof of Proposition
\ref{p:lin} with $A=Df_0$.
The parameter $\rho$ has to be chosen sufficiently small.
It controls the size of the cone in  Lemma \ref{l:cone1}.
We always take $\rho>0$ so that
\begin{align*}
&\lambda_i<\lambda_j-\rho \quad\hbox{when $\lambda_i<\lambda_j$,}\\
&\lambda_i-\rho>1 \quad\hbox{when $\lambda_i>1$,}\\
&\lambda_i+\rho<1 \quad\hbox{when $\lambda_i<1$.}
\end{align*}
Note that since $(Df_0)|_{W^{min}_{f_0}}$ is semisimple, we also have
\begin{align*}
\|(Df_0)v\|= \lambda_1\|v\|,\quad v\in \hbox{im}(P_1),
\end{align*}
and
$$
\|(Df_0)^{-n}\|\le \lambda_1^{-n}.
$$
By the assumption on $g_k$'s, there exists $\eta>0$ such that
\begin{align}\label{eq:esss}
\max_k\|P_1(Dg_k)v\|>\eta\|v\|,\quad v\in\hbox{Lie}(G).
\end{align}

Let $\mu_i=\lambda_1^{-1}(\lambda_i+\rho)$ and $\sigma_i=\frac{\log \mu_i}{\log \mu_d}$. 
For $v\in\hbox{Lie}(G)$, we define
$$
N(v)=\max_{i>1}\left\{\|P_iv\|^{\sigma_i^{-1}}\right\}.
$$

For $\beta,s>0$, we define
$$
C(\beta,s)=\{v\in\hbox{Lie}(G):\, N(v)\le \beta \|v\|^s\}.
$$

\begin{lemma} \label{l:cone1}
There exist $\epsilon,\beta>0$
such that for every $x,y\in X$ satisfying $d(x,y)<\epsilon$ and (\ref{eq:cb}),
$$
y\in\exp(C(\beta,s))x.
$$
where $s=s(\nu,\rho,\vartheta,f_0)>0$ is such that  $s\to\infty$ as $\nu\to 1^+$ and $\rho\to 0^+$.
\end{lemma}

\begin{proof}
Let $c_1= \max_k \|Dg_k\|$.

There exist $\delta_0>0$ and $c_0>1$ such that for every $x\in X$
and $v\in\hbox{Lie}(G)$ satisfying $\|v\|<\delta_0$, we have
\begin{equation}\label{eq:ccc}
c_0^{-1}\|v\|\le d(x,\exp(v)x)\le c_0\|v\|.
\end{equation}
Let  $b>0$ such that $\sum_{j>1} b^{\sigma_j}=\delta_0/(2c_1)$.
We choose $\epsilon>0$ so that 
$d(x,y)<\epsilon$ implies that $y=\exp(v)x$ where 
$$
N(v)<\min\{1,b\}\quad\hbox{and}\quad \|v\|<\min\left\{\delta_0,\delta_0/2c_1\right\}.
$$
Assuming that the claim fails, we will show that there exists $n\ge 0$ such that
\begin{align}\label{eq:cm}
a c_0^{\vartheta+1}\nu^n\|v\|^\vartheta<\max_k\|D(f_0^{-n}g_kf_0^n)v\|<\delta_0.
\end{align}
Since
$$
d(f_0^{-n}g_kf_0^{n}(x),f_0^{-n}g_kf_0^{n}(y))=d(f_0^{-n}g_kf_0^{n}(x),\exp(D(f_0^{-n}g_kf_0^{n})v)f_0^{-n}g_kf_0^{n}(x)),
$$
we deduce from (\ref{eq:ccc}) and (\ref{eq:cm}) that
$$
ac_0^{\vartheta+1}\nu^n(c_0^{-1}\, d(x,y))^\vartheta< \max_k c_0\, d(f_0^{-n}g_kf_0^{n}(x),f_0^{-n}g_kf_0^{n}(y)),
$$
which contradicts (\ref{eq:cb}).

To obtain the upper estimate in (\ref{eq:cm}), we observe that
\begin{align*}
\max_k\|D(f_0^{-n}g_kf_0^n)v\|&\le \lambda_1^{-n}\max_k\|D(g_kf_0^n)v\|\le \lambda_1^{-n}c_1\|D(f_0^n)v\|\\
&\le c_1\|P_1v\| + \lambda_1^{-n} c_1 \sum_{j>1}  (\lambda_j+\rho)^n \|P_jv\|\\
&\le  c_1\|v\|+ c_1\sum_{j>1} \mu_j^n\|P_jv\|.
\end{align*}
We choose $n\ge 0$ so that
\begin{equation}\label{eq:mu_d}
\mu_d^{-1} \frac{b}{N(v)} <\mu_d^n\le \frac{b}{N(v)}.
\end{equation}
Then 
$$
\mu_d^{-\sigma_j} \frac{b^{\sigma_j}}{N(v)^{\sigma_j}}<\mu_j^n\le \frac{b^{\sigma_j}}{N(v)^{\sigma_j}}
$$
and  
\begin{align*}
\max_k\|D(f_0^{-n}g_kf_0^n)v\|\le c_1\|v\|+c_1 \sum_{j>1} b^{\sigma_j}\frac{\|P_jv\|}{N(v)^{\sigma_j}}
<\delta_0.
\end{align*}
The lower estimate in (\ref{eq:cm}) is proved similarly using that $g_1,\ldots,g_l$ is essential (see (\ref{eq:esss})).
Let $\gamma_j>0$ be such that $\lambda_1^{-1}(\lambda_j-\rho)=\mu_j^{1-\gamma_j}$. We have
\begin{align*}
\max_k\|D(f_0^{-n}g_kf_0^n)v\|&\ge \max_k \lambda_1^{-n}\|P_1 D(g_kf_0^n)v\|\ge \lambda_1^{-n}\eta \|D(f_0^n)v\|\\
&\ge \lambda_1^{-n}\eta\left(\lambda_1^n\|P_1v\|+ \sum_{j>1}(\lambda_j-\rho)^n\|P_jv\|\right)\\
&\ge \eta\sum_{j>1} \mu_j^{n(1-\gamma_j)} \left\| P_jv \right\|\\
&\ge \eta \sum_{j>1} (\mu_d^{-1}b)^{\sigma_j(1-\gamma_j)} N(v)^{\sigma_j\gamma_j}\frac{\|P_jv\|}{N(v)^{\sigma_j}}\\
&\ge \eta (\mu_d^{-1}b)^{\sigma_{j_0}(1-\gamma_{j_0})}N(v)^{\sigma_{j_0}\gamma_{j_0}}.
\end{align*}
where $j_0>1$ is such that $\|P_{j_0}v\|^{1/\sigma_{j_0}}=N(v)$.
This implies that
$$
\max_k\|D(f_0^{-n}g_kf_0^n)v\| \ge \min_{j>1} \eta (\mu_d^{-1}b)^{\sigma_{j}(1-\gamma_{j})}N(v)^{\sigma_{j}\gamma_{j}}.
$$
Let $\omega=\frac{\log\nu}{\log\mu_d}$.
It follows from (\ref{eq:mu_d}) that the first inequality in (\ref{eq:cm}) is satisfied provided that
$$
a c_0^{\vartheta+1}N(v)^{-\omega}b^{\omega}\|v\|^\vartheta<\min_{j>1} \eta (\mu_d^{-1}b)^{\sigma_{j}(1-\gamma_{j})}N(v)^{\sigma_{j}\gamma_{j}}.
$$
Since this gives a contradiction, we deduce that
$$
a c_0^{\vartheta+1}b^{\omega}\|v\|^\vartheta\ge \min_{j>1} \eta (\mu_d^{-1}b)^{\sigma_{j}(1-\gamma_{j})}N(v)^{\omega+\sigma_{j}\gamma_{j}}.
$$
Hence,
$$
N(v)\le \beta \|v\|^s
$$
with explicit $\beta>0$ and $s=\vartheta/(\omega+\max_{j>1} (\sigma_j\gamma_j))$.
Clearly, $s\to\infty$ as $\nu\to 1^+$ and $\rho\to 0^+$.
This completes the proof.
\end{proof}

For $i=1,\ldots,d$ and $\delta,\beta,s>0$, we define
$$
C^i_\delta(\beta,s)=\{v\in\hbox{Lie}(G):\, \|v\|<\delta;\, \|P_iv\|^{\sigma_i^{-1}}\le \beta \|v\|^{s};\, P_jv=0, j>i\}.
$$
%Note that $C^d_\delta(\beta,s)=C(\beta,s)\cap \{\|v\|<\delta\}$.

\begin{lemma}  \label{l:cone2}
For every $\delta,\beta,s>0$,
$$
(Df_0)^{-1} (C^i_\delta(\beta,s))\subset C^i_{\xi\delta}(\rho_i\beta,s).
$$
where $\xi=\max\{1,\|(Df_0)^{-1}\|\}$ and $\rho_i=(\lambda_i-\rho)^{-{\sigma_i^{-1}}}(\lambda_i+\rho)^{s}$.
\end{lemma}

\begin{proof}
Let $v\in (Df_0)^{-1} (C^i_\delta(\beta,s))$. Then
\begin{align*}
(\lambda_i-\rho)^{\sigma_i^{-1}}\|P_iv\|^{\sigma_i^{-1}}\le\beta\left(\sum_{j\le i} (\lambda_j+\rho)\|P_jv\|\right)^{s}\le \beta(\lambda_i+\rho)^{s}\|v\|^{s}.
\end{align*}
This implies the lemma.
\end{proof}

\begin{proof}[Proof of Theorem \ref{th:quant}]
We start by setting up notation for the Jordan form of $Df_0$ for $\lambda_i=1$.
It follows from our choice of the norm that there exist linear maps $Q_1,\ldots, Q_{j_0}$ such that
\begin{equation}\label{eq:qq_i}
\|(Df_0^k)v\|=\left\|\sum_{j=0}^{j_0} k^j Q_jv\right\| \quad\hbox{for $k\ge 0$ and $v\in\hbox{im}(P_i)$.}
\end{equation}

Let $s>0$ be as in Lemma \ref{l:cone1}.
Recall that $s\to\infty$ as $\nu\to 1^+$ and $\rho\to 0^+$.
We choose $\rho>0$ and $\nu> 1$ so that
\begin{align*}
& s-\sigma_i^{-1}>0\quad\hbox{when $\lambda_i=1$},\\
& \rho_i:=(\lambda_i-\rho)^{-{\sigma_i^{-1}}}(\lambda_i+\rho)^{s}<1\quad\quad\hbox{when $\lambda_i<1$.}
\end{align*}
Let $\xi\ge 1$ be as in Lemma \ref{l:cone2} and $\beta,\epsilon >0$ as in Lemma \ref{l:cone1}.
Take $\delta\in (0,1)$ such that for $\|v\|<\xi\delta$,
the exponential coordinates $v\mapsto \exp(v)z$, $z\in X$, are one-to-one,
and 
\begin{equation} \label{eq:q_j}
\hbox{$\|Q_j P_i v\|<\beta^{-(s-\sigma_i^{-1})^{-1}}$ when $\lambda_i=1$ and $j=0,\ldots,j_0$.}
\end{equation}
In addition, we assume that $\epsilon$ is sufficiently small so that 
$$
B_\epsilon(x)\subset \exp(\{\|v\|<\delta\})x\quad\hbox{for all $x\in X$.}
$$
Then by Lemma \ref{l:cone1},
\begin{equation}\label{eq:cone}
\mathcal{L}_\epsilon(x)\subset \exp(C(\beta,s)\cap \{\|v\|<\delta\})x\quad\hbox{for every $x\in X$.}
\end{equation}
 In particular,
\begin{equation}\label{eq:cone3}
\mathcal{L}_\epsilon(x)\subset \exp(C^d_{\delta}(\beta,s))x.
\end{equation}

If $\lambda_d\le 1$, we argue as in the following paragraph. Otherwise, we observe that
since $\delta<1$, we have
$$
C^d_{\delta}(\beta,s_1)\subset C^d_{\delta}(\beta,s_2)\quad\hbox{for $s_1>s_2$},
$$
and hence inclusion (\ref{eq:cone}) also holds for $s>0$ such that
$\rho_d=(\lambda_d-\rho)^{-{\sigma_d^{-1}}}(\lambda_d+\rho)^{s}<1$.
Applying $f_0^{-1}$ to (\ref{eq:cone3}), we deduce from Lemma \ref{l:cone2} that 
\begin{equation}\label{eq:rho_d}
\mathcal{L}_\epsilon(x)\subset \exp(C^d_{\xi\delta}(\rho_d\beta,s))x
\end{equation}
for every $x\in X$. Using that the exponential coordinates are one-to-one,
we obtain from (\ref{eq:rho_d}) and (\ref{eq:cone}) that
$$
\mathcal{L}_\epsilon(x)\subset \exp(C^d_{\delta}(\rho_d\beta,s))x.
$$
Repeating this argument, we conclude that
$$
\mathcal{L}_\epsilon(x)\subset \bigcap_{k\ge 1}
\exp(C^d_{\delta}(\rho^k_d\beta,s))x=
\exp(C^d_{\delta}(0,s))x.
$$
Now (\ref{eq:cone}) implies that
$$
\mathcal{L}_\epsilon(x)\subset \exp(C^{d-1}_{\delta}(\beta,s))x.
$$
Applying the same reasoning inductively on $i$, we deduce that
$$
\mathcal{L}_\epsilon(x) \subset
\exp(C^i_{\delta}(0,s))x
$$
provided that $\lambda_i>1$. It follows from (\ref{eq:cone}) that
$\mathcal{L}_\epsilon(x)\subset \exp(C^{i-1}_{\delta}(\beta,s))x$.

Suppose $\lambda _i=1$ and $\mathcal{L}_\epsilon(x) \subset \exp (C^i_\delta(\beta,s) ) x$
for some $\beta>0$. We will show that 
$$
\mathcal{L}_\epsilon (x) \subset \exp  (C^i _\delta(0,s))x.
$$
Applying $f_0^{-1}$, we deduce that
for $y=\exp(v)x\in\mathcal{L}_\epsilon(x)$, $\|v\|<\delta$, and $k\ge 0$, we have  
$$
\|(Df_0^k)P_iv\|^{\sigma_i^{-1}}\le \beta \left(\sum_{j< i} (\lambda_j+\rho)^k\|P_jv\|+ \|(Df_0^k)P_iv\|\right)^{s}.
$$
Using that  $\lambda_j+\rho<1$ for $j<i$ and taking $k\to\infty$, we deduce from (\ref{eq:qq_i}) that
$$
\|Q_{j_0}P_iv\|^{\sigma_i^{-1}}\le \beta \|Q_{j_0}P_iv\|^{s}.
$$
By the choice of $\delta$ (see (\ref{eq:q_j})), $\|Q_{j_0}P_iv\|=0$.
Similar arguments imply that $\|Q_jP_iv\|=0$ for all $j=0,\ldots,j_0$.
Hence, $P_iv=0$ and $\mathcal{L}_\epsilon(x)\subset \exp(C^{i}_{\delta}(0,s))x$.
Combining this estimate with (\ref{eq:cone}), we deduce that
$\mathcal{L}_\epsilon(x)\subset \exp(C^{i-1}_{\delta}(\beta,s))x$.

Now we consider the case when $\mathcal{L}_\epsilon(x)\subset \exp(C^i_{\delta}(\beta,s))$
for some $i$ such that $\lambda_i<1$ and $\beta>0$.
Applying $f_0^{-1}$, it follows from Lemma \ref{l:cone2} that
$$
\mathcal{L}_\epsilon(x)\subset \exp(C^i_{\xi\delta}(\rho_i\beta,s))x\quad\hbox{for every $x\in X$.}
$$
Then it follows from (\ref{eq:cone}) that
$$
\mathcal{L}_\epsilon(x)\subset \exp(C^i_{\delta}(\rho_i\beta,s))x,
$$
and repeating this argument, we deduce that
$$
\mathcal{L}_\epsilon(x)\subset \bigcap_{k\ge 1}
\exp(C^i_{\delta}(\rho^k_i\beta,s))x=
\exp(C^i_{\delta}(0,s))x.
$$
Since the above argument can be applied inductively on $i$,
and we conclude that $\mathcal{L}_\epsilon(x)\subset \exp(C^2_{\delta}(0,s))x$.
This completes the proof.
\end{proof}

\begin{proof}[Proof of Theorem \ref{th:ffs}]
The first claim of Theorem \ref{th:ffs} follows from
Theorem \ref{th:quant} with $\mathcal{L}_\epsilon(x)=\mathcal{S}_{\epsilon'}(x)$
where $\mathcal{S}_{\epsilon'}(x)$ is as in (\ref{eq:se}) with sufficiently small $\epsilon'>0$.
Note that $\alpha_0(\Gamma)$ contains an essential subset by Lemma \ref{l:ess},
and (\ref{eq:cb}) follows from (\ref{eq:ff_0}) where the parameter $\nu$
is close to one if $f$ and $f_0$ are $C^1$-close.

It remains to show that $\Phi$ is bi-H\"older with respect to the metrics $d^{fs}$.
There exists $\epsilon>0$ such that for every $x\in X$, any points $z,w\in X$
lying on the same local leaf of $W^{fs}_{f}$ in $B(x,\epsilon)$ satisfy (\ref{eq:fs_s}).
Let $\delta>0$ be such that $\Phi(B_\delta(y))\subset B_\epsilon(\Phi(y))$ for every $y\in X$.
Consider points $z_0,w_0\in X$ lying on the same leaf of  $W^{fs}_{f_0}$ such that $d^{fs}(z_0,w_0)<\delta$.
Let $\ell$ be a curve from $z_0$ to $w_0$ contained in $W^{fs}_{f_0}(z_0)$ such that $L(\ell)=d^{fs}(z_0,w_0)$.
Then $\Phi(\ell)$ is contained in $B_\epsilon(\Phi(z_0))\cap W^{fs}_f(\Phi(z_0))$. Moreover, since
$\Phi(\ell)$ is connected, $\Phi(\ell)$ is contained in a single local leaf of  $W^{fs}_f$ in
$B_\epsilon(\Phi(z_0))$. Hence, 
$$
d^{fs}(\Phi(z_0),\Phi(w_0))\ll d(\Phi(z_0),\Phi(w_0)).
$$
Since $\Phi$ is H\"older with respect to $d$, this implies that $\Phi$ is H\"older with respect
to $d^{fs}$ as well. The proof that $\Phi^{-1}$ is H\"older with respect to $d^{fs}$ is similar.
\end{proof}

\subsection{Convergence of the sequences $f^{-n}gf^n$} \label{s3}

In this section, we study convergence of the sequence of maps $f^{-n}gf^n$ as $n\to\infty$.

%The crucial ingredient of the proof is show that for a $C^1$-small perturbations $f$ and $g$ of
%a good pair $f_0,g_0\in\Aff(X)$ and  a.e. $x\in X$, the sequence of maps
%$(f^{-n}gf^n)_{W^{fs}_f(x)}$ is precompact in $C^1$-topology.

First, we consider the algebraic setting:

\begin{prop}\label{p:alg}
Let $f_0,g_0\in\Aff(X)$ be such that $Df_0:W_{f_0}^{min}\to W_{f_0}^{min}$ is semisimple. Then
\begin{enumerate}
\item Given a sequence $\{m_i\}$ such that
$$
(f_0^{-m_{i}}g_0f_0^{m_{i}})(x)\to y\quad\hbox{as $i\to\infty$}
$$
for some $x,y\in X$, the sequence of maps $f_0^{-m_{i}}g_0f_0^{m_{i}}:{W^{fs}_{f_0}(x)}\to X$
is precompact in the $C^0$-topology.
\item 
There exist a sequence $\{n_i\}$ and a linear map $A:W_{f_0}^{min}\to W_{f_0}^{min}$ such that
if  for some $x,y\in X$ and a subsequence $\{n_{i_j}\}$,
$$
(f_0^{-n_{i_j}}g_0f_0^{n_{i_j}})(x)\to y\quad\hbox{as $j\to\infty$,}
$$
then uniformly on $v\in W^{min}_f$ in compact sets,
$$
(f_0^{-n_{i_j}}g_0f_0^{n_{i_j}})\exp(v)x\to \exp(Av)y\quad\hbox{as $j\to\infty$}.
$$

The map $A$ is nondegenerate provided that $P_{f_0}^{min}Dg_0:W_{f_0}^{min}\to W_{f_0}^{min}$
is nondegenerate.
\end{enumerate}
\end{prop}

\begin{remark}{\rm \label{r:alg}
If $\dim W_{f_0}^{min}=1$, one can take $n_i=i$ and $A=P_{f_0}^{min}Dg_0$.
In general, $A=\lim_{i\to\infty} \omega^{-n_i} P^{min}_{f_0}(Dg_0) \omega^{n_i}$
for some $\omega\in \hbox{Isom}(W^{min}_{f_0})$.
}
\end{remark}

\begin{proof}
We have
$$
(f_0^{-n}g_0f_0^{n})\exp(v)x=\exp(D(f_0^{-n}g_0f_0^{n}) v)(f_0^{-n}g_0f_0^{n})x.
$$
It follows from the assumption on $f_0$ that
$$
Df_0|_{W_{f_0}^{min}}=\lambda\cdot \omega
$$
where $\lambda>0$ and $\omega$ is an isometry of $W_{f_0}^{min}$.
Then
\begin{align*}
D(f_0^{-n}g_0f_0^{n}) v= ( \omega^{-n} P^{min}_{f_0}(Dg_0) \omega^{n})v+(Df_0)^{-n}P^{max}_{f_0}(Dg_0)\lambda^{n}\omega^{n}v
\end{align*}
where $P^{min}_{f_0}$ denotes the projection on $W^{min}_{f_0}$ and $P^{max}_{f_0}$ denotes
the projection on the sum of eigenspaces complimentary to $W^{min}_{f_0}$. Since
$\omega$ is an isometry, and
$$
(Df_0)^{-n}P^{max}_{f_0}(Dg_0)\lambda^{n}\omega^{n}v\to 0,
$$
it is clear that the sequence of maps $v\mapsto D(f_0^{-n}g_0f_0^{n}) v$ is precompact in $C^0$-topology.
This implies that the sequence $f_0^{-m_{i}}g_0f_0^{m_{i}}|_{W^{fs}_{f_0}(x)}$ is precompact in
$C^0$-topology as well. 

To prove (2), it suffices to choose the sequence $\{n_{i}\}$ so that $\{\omega^{n_{i}}\}$ converges.
This proves the proposition.
\end{proof}

We show that the convergence of $f_0^{-n}g_0f_0^n|_{W^{fs}_{f_0}(x)}$ persists under small perturbations:

\begin{thm}\label{th:compact2}
Let $f_0,g_0\in\Aff(X)$ satisfy
\begin{enumerate}
\item[(i)] The map $f_0$ is partially hyperbolic,
\item[(ii)] The map $Df_0 :W^{min}_{f_0}\to W^{min}_{f_0}$ is semisimple.
%\item[(iii)] The map $P^{min}_{f_0} (Dg_0):W^{min}_{f_0}\to W^{min}_{f_0}$ is nondegenerate.
\end{enumerate}
Let $f,g\in \Diff(X)$ be $C^1$-small perturbations of $f_0$ and $g_0$
and  $\Phi:X\to X$ a H\"older isomorphism such that
$$
\Phi\circ f_0=f\circ \Phi\quad\hbox{and}\quad \Phi\circ g_0=g\circ \Phi
$$
and 
$$
\Phi(W^{fs}_{f_0}(x))=W^{fs}_{f}(\Phi(x))\quad\hbox{for every $x\in X$.}
$$
Then for every $x\in X$ and a sequence $\{m_i\}$ as in Proposition \ref{p:alg}(1), the sequence of maps
$$
f^{-m_i}gf^{m_i}: W^{fs}_f(x)\to X,\quad i\ge 0,
$$
is precompact in the $C^1$-topology.
\end{thm}

Throughout this section, we assume that $X$ is a submanifold of $\mathbb{R}^N$,
which allows us to identify tangent spaces at different points.

We have a H\"older continuous decomposition (cf. (\ref{eq:hyp}))
\begin{equation}\label{eq:split2}
T_x X= E^-_x\oplus E^+_x,\quad x\in X,
\end{equation}
where $E^-_x=T_x W^{fs}_f(x)$.
Let
$$
P_x: T_x X\to E^-_x\quad\hbox{and}\quad P^+_x: T_x X\to E^+_x
$$
denote the corresponding projections.

The following proposition is the main ingredient of the proof of Theorem \ref{th:compact2}.

\begin{prop}\label{eq:main_est}
Let $r>0$. Then under the assumptions of Theorem \ref{th:compact2}, for every $x,y\in X$ satisfying $y\in
W^{fs}_f(x)$ and $d^{fs}(x,y)\le r$,
$$
\|D(f^{-n}gf^n)_xP_x-D(f^{-n}gf^n)_y P_y\|\ll d^{fs}(x,y)^\vartheta\|D(f^{-n}gf^n)_xP_x\|+\delta_n
$$
where $\vartheta>0$ and $\delta_n\to 0$.
\end{prop}

\begin{proof}
Note that $\Phi$ and $\Phi^{-1}$ are also H\"older with respect to the metrics $d^{fs}$ on the fast
stable leaves of $f_0$ and $f$ (see proof of Theorem \ref{th:ffs}).
By Proposition \ref{p:fs0},
\begin{align*}
d(f_0^{-k}g_0f_0^n(\Phi^{-1}(x)), f_0^{-k}g_0f_0^n(\Phi^{-1}(y))) &\ll \lambda_0^{n-k}
d^{fs}(\Phi^{-1}(x),\Phi^{-1}(y))\\
&\ll\lambda_0^{n-k} d^{fs}(x,y)^{\omega_0}
\end{align*}
where $\omega_0>0$ is the H\"older exponent of $\Phi^{-1}$ with respect to $d^{fs}$.
Then it follows that we have the estimate
\begin{equation}\label{eq:bb}
d(f^{-k}gf^n(x), f^{-k}gf^n(y))\ll \lambda_0^{\omega(n-k)} d^{fs}(x,y)^{\omega_0\omega}
\end{equation}
where $\omega>0$ is the H\"older exponent of $\Phi$ with respect to $d$.

Since the decomposition (\ref{eq:split2}) is $f$-invariant, we have
$$
P_{f(x)}D(f)_xP_x= D(f)_xP_x\quad\hbox{and}\quad
P_{f^{-1}(x)}D(f^{-1})_x P_x= D(f^{-1})_xP_x.
$$
By (\ref{eq:hyp2}),
there exist $\lambda\in (0,1)$ and $\mu>\lambda$ such that
\begin{equation}\label{eq:fast_stable2}
\|D(f^n)_x P_x\|\ll\lambda^n\quad\hbox{and}\quad
\|D(f^{-n})_x P^+_x\|\ll \mu^{-n}
\end{equation}
uniformly on $x\in X$ and $n\ge 0$.
It is crucial for the proof that the map $D(f)_xP_x$ is approximately conformal (cf. assumption (ii) on
$f_0$).
Namely, for some small $\epsilon>0$,
\begin{equation}\label{eq:fast_stablea}
\|D(f^{-n})_x P_x\|\ll (\lambda-\epsilon)^{-n}
\end{equation}
uniformly on $x\in X$ and $n\ge 0$. We also recall that for $\rho> \lambda$
and $x,y\in X$ such that $y\in W^{fs}_f(x)$,
\begin{equation}\label{eq:rhoo}
d^{fs}(f^n(x),f^n(y))\ll \rho^n d^{fs}(x,y).
\end{equation}
Note that the parameter $\epsilon$ in (\ref{eq:fast_stablea}) satisfies
$\epsilon\to 0$ as  $d_{C^1}(f_0,f)\to 0$. 
We assume $f$ is sufficiently close to $f_0$
so that
$$
\zeta:=(\lambda-\epsilon)^{-1}\lambda \rho ^{\theta}<1\quad\hbox{and}\quad
\nu:=(\lambda-\epsilon)^{-1}\lambda\lambda_0^{\omega}<1
$$
where $\theta$ is the H\"older exponent of the map $x\mapsto P_x$.

We have
\begin{align*}
D(f^{-n}gf^n)_x P_x=&D(f^{-n})_{gf^n(x)} P_{gf^n(x)}
D(g)_{f^n(x)}D(f^n)_x P_x\\
&+D(f^{-n})_{gf^n(x)} P^+_{gf^n(x)} D(g)_{f^n(x)}D(f^n)_xP_x.
\end{align*}
It follows from (\ref{eq:fast_stable2}) that
$$
\|D(f^{-n})_{gf^n(x)} P^+_{gf^n(x)} D(g)_{f^n(x)}D(f^n)_xP_x\|\ll \lambda^n\mu^{-n}\to 0.
$$
Hence, to prove the theorem, it suffices to show that for 
$$
A_n(x):=\left(\prod_{i=n-1}^0  D(f^{-1})_{f^{-i}gf^n(x)}\right)
P_{gf^n(x)} D(g)_{f^n(x)}
\left(\prod_{i=n-1}^0  D(f)_{f^i(x)}\right)P_x,
$$
we have
$$
\|A_n(x)-A_n(y)\|\ll d^{fs}(x,y)^\kappa\|A_n(x)\|.
$$

We consider the operators
\begin{align*}
A_{n,k}(x,y):=&\left(\prod_{i=n-1}^0  D(f^{-1})_{f^{-i}gf^n(x)}\right)
P_{gf^n(x)}D(g)_{f^n(x)}\\
&\times
\left(\prod_{i=n-1}^{k+1}  D(f)_{f^i(x)}\right)P_{f^{k+1}(x)}\left(\prod_{i=k}^0  D(f)_{f^i(y)}\right)P_y.
\end{align*}
Note that
\begin{align}\label{eq:n-1}
\|A_n(x)-A_{n,-1}(x,y)\|\le \|A_n(x)\|\cdot \|P_x-P_xP_y\|\ll\|A_n(x)\| d(x,y)^\theta.
\end{align}
Now we estimate $\|A_{n,n-1}(x,y)-A_{n,-1}(x,y)\|$. We use that
\begin{align*}
A_{n,k}(x,y)-A_{n,k-1}(x,y)=A_{n}(x)B_{n,k}(x,y)
\end{align*}
where
\begin{align*}
B_{n,k}(x,y):=&\left(\prod_{i=0}^{k}  D(f)^{-1}_{f^{i}(x)}\right)P_{f^{k+1}(x)}
\left(D(f)_{f^{k}(y)}P_{f^k(y)}-D(f)_{f^{k}(x)}P_{f^k(x)}\right)\\
&\times \left(\prod_{i=k-1}^0  D(f)_{f^i(y)}\right)P_y.
\end{align*}
By (\ref{eq:rhoo}), we have
$$
\|D(f)_{f^{k}(y)}P_{f^k(y)}-D(f)_{f^{k}(x)}P_{f^k(x)}\|\ll
d(f^{k}(x),f^{k}(y))^\theta \ll \rho^{\theta k} d^{fs}(x,y)^\theta,
$$
and by (\ref{eq:fast_stable2}) and (\ref{eq:fast_stablea}),
\begin{align*}
\left\|\left(\prod_{i=k-1}^0  D(f)_{f^i(y)}\right)P_y\right\|&\ll \lambda^k,\\
\left\|\left(\prod_{i=0}^{k}  D(f)^{-1}_{f^{i}(x)}\right)P_{f^{k+1}(x)}\right\|&\ll (\lambda-\epsilon)^{-k-1}.
\end{align*}
Hence,
$$
\|B_{n,k}(x,y)\|\ll \zeta^k d^{fs}(x,y)^\theta
$$
Since $\zeta<1$, it follows that
\begin{align}\label{eq:a}
\|A_{n,n-1}(x,y)-A_{n,-1}(x,y)\|&\le \sum_{k=0}^{n-1}\|A_{n,k}(x,y)-A_{n,k-1}(x,y)\|\\
&\ll \|A_{n}(x)\|d^{fs}(x,y)^\theta.\nonumber
\end{align}

We claim that for some $c>0$ and all $k=-1,\ldots,n-1$,
\begin{equation}\label{eq:c00}
\|A_{n,k}(x,y)\|\ll  (1+ c\,d^{fs}(x,y)^\theta)\cdot \|A_{n}(x)\|.
\end{equation}
Setting
$$
C_k(x,y):= \left(\prod_{i=0}^{k}  D(f)^{-1}_{f^i(x)}\right)P_{f^{k+1}(x)} \left(\prod_{i=k}^0  D(f)_{f^i(y)}\right)P_y,
$$
we have
$$
A_{n,k}(x,y)=A_{n}(x)C_k(x,y).
$$
Now equation (\ref{eq:c00}) will follow from the estimate
$$
\|C_k(x,y)\|\ll 1+c\, d^{fs}(x,y)^\theta.
$$
In fact, we will show that
\begin{equation}\label{eq:c}
\|C_k(x,y)-P_xP_y\|\ll d^{fs}(x,y)^\theta.
\end{equation}
Using (\ref{eq:fast_stable2}) and (\ref{eq:fast_stablea}), we deduce that
\begin{align*}
&\|C_k(x,y)-C_{k-1}(x,y)\|\\
= &
 \left\|
\left( \prod_{i=0}^{k-1} D(f)^{-1}_{f^i(x)} \right)P_{f^{k}(x)}
\left( D(f)^{-1}_{f^{k}(x)} D(f)_{f^k(y)}-id \right)\right.\\
& \quad\quad\left. \times
  \left(\prod_{i=k-1}^0  D(f)_{f^i(y)}\right)P_y
\right\|\\
\ll & (\lambda-\epsilon)^{-k} d(f^k(x),f^k(y))^\theta \lambda^k\ll \zeta^k d^{fs}(x,y)^\theta.
\end{align*}
Since $C_{-1}(x,y)=P_xP_y$ and $\zeta<1$, the last estimate implies (\ref{eq:c}) and (\ref{eq:c00}).
%Hence, it follows from (\ref{eq:a}) that
%\begin{align}\label{eq:step1}
%\|A_{n,n-1}(x,y)-A_{n,-1}(x,y)\|\ll (1+c\cdot d(x,y)^\theta)d(x,y)^\theta\|A_{n,-1}(x,y)\|.
%\end{align}
%Combining this estimate with (\ref{eq:n-1}), we conclude that
%\begin{equation}\label{eq:n}
%\|A_{n}(x)-A_{n,n-1}(x,y)\|\ll d(x,y)^\theta\|A_{n}(x)\|.
%\end{equation}

Next, we consider the operators
\begin{align*}
D_{n,k}(x,y):=&\left(\prod_{i=n-1}^{k}  D(f^{-1})_{f^{-i}gf^n(y)}\right)P_{f^{-k}gf^n(y)}
\left(\prod_{i=k-1}^0  D(f^{-1})_{f^{-i}gf^n(x)}\right)\\
& \times P_{gf^n(x)} D(g)_{f^n(x)} P_{f^n(x)}
\left(\prod_{i=n-1}^0  D(f)_{f^i(y)}\right)P_y.
\end{align*}
Using (\ref{eq:bb}), we deduce that
\begin{align*}
\|A_{n,n-1}(x,y)-D_{n,n}(x,y)\|
&\le \|P_{f^{-n}gf^n(x)}-P_{f^{-n}gf^n(y)}P_{f^{-n}gf^n(x)}\|\cdot \|A_{n,n-1}(x,y)\|\\
&\ll d(f^{-n}gf^n(x),f^{-n}gf^n(y))^\theta  \|A_{n,n-1}(x,y)\|\\
&\ll d^{fs}(x,y)^{\theta\omega_0\omega} \|A_{n,n-1}(x,y)\|\\
&\ll d^{fs}(x,y)^{\theta\omega_0\omega} \|A_{n}(x)\|.
\end{align*}
To estimate $\|D_{n,n}(x,y)-D_{n,0}(x,y)\|$, we use the argument similar to the proof
of (\ref{eq:a}). We have 
$$
D_{n,k}(x,y)-D_{n,k-1}(x,y)= E_{n,k}(x,y)A_{n,n-1}(x,y)
$$
where
\begin{align*}
E_{n,k}(x,y):=&\left(\prod_{i=n-1}^{k}  D(f^{-1})_{f^{-i}gf^n(y)}\right)P_{f^{-k}gf^n(y)}\\
 &\times \left(D(f^{-1})_{f^{-(k-1)}gf^n(x)} P_{f^{-(k-1)}gf^n(x)}-
 D(f^{-1})_{f^{-(k-1)}gf^n(y)} P_{f^{-(k-1)}gf^n(y)} \right) \\
&\times \left(\prod_{i=k-1}^{n-1}  D(f^{-1})^{-1}_{f^{-i}gf^n(x)}\right) P_{f^{-n}gf^n(x)} 
\end{align*}
Applying (\ref{eq:fast_stablea}), (\ref{eq:bb}), and (\ref{eq:fast_stable2}), we deduce
that
$$
\|E_{n,k}(x,y)\|\ll \nu^{n-k} d^{fs}(x,y)^{\theta\omega_0\omega}.
$$
Since $\nu<1$, it follows that
\begin{align}\label{eq:step1}
\|D_{n,n}(x,y)-D_{n,0}(x,y)\|
&\le \sum^n_{k=1} \|D_{n,k}(x,y)-D_{n,k-1}(x,y)\|\\
&\ll d^{fs}(x,y)^{\theta\omega_0\omega} \|A_{n,n-1}(x,y)\|\ll d^{fs}(x,y)^{\theta\omega_0\omega} \|A_{n}(x)\|.
\end{align}

Next, we compare the maps $A_n(y)$ and $D_{n,0}(x,y)$:
\begin{align*}
\|A_n(y)-D_{n,0}(x,y)\|
= &\left\|
\left(\prod_{i=n-1}^{0}  D(f^{-1})_{f^{-i}gf^n(y)}\right)P_{gf^n(y)}\right.\\
&\quad\quad\times \left(P_{gf^n(y)} D(g)_{f^n(y)} P_{f^n(y)}-P_{gf^n(x)} D(g)_{f^n(x)} P_{f^n(x)}\right)\\
&\left.\quad\quad\times \left(\prod_{i=n-1}^0  D(f)_{f^i(y)}\right)P_y
\right\|.
\end{align*}
We have
\begin{align*}
\left\|
P_{gf^n(y)} D(g)_{f^n(y)} P_{f^n(y)}-P_{gf^n(x)} D(g)_{f^n(x)} P_{f^n(x)}
\right\|&\ll d(f^n(x),f^n(y))^\theta\\
&\ll \rho^{\theta n} d^{fs}(x,y)^\theta.
\end{align*}
Combining this estimate with (\ref{eq:fast_stable2}) and (\ref{eq:fast_stablea}),
we deduce that
$$
\|A_n(y)-D_{n,0}(x,y)\|\ll \zeta^n d^{fs}(x,y)^\theta.
$$

Finally, the proposition follows from the estimate
\begin{align*}
\|A_n(x)-A_n(y)\|\le&
\|A_n(x)-A_{n,-1}(x,y)\|+\|A_{n,-1}(x,y)-A_{n,n-1}(x,y)\|\\
&+\|A_{n,n-1}(x,y)-D_{n,n}(x,y)\|+
\|D_{n,n}(x,y)-D_{n,0}(x,y)\|\\
&+\|D_{n,0}(x,y)-A_n(y)\|.
\end{align*}
This completes the proof.
\end{proof}

\begin{prop}\label{p:bound}
Let $x_0\in X$ and $r>0$. Then
under the assumptions of Theorem \ref{th:compact2},
%the set of maps $\{f^{-n}gf^n|_{W^{fs}_f(x)}\}$ is uniformly H\"older. Then
$$
\sup\{\|D(f^{-n}gf^n)_xP_x\|:\,\, x\in W^{fs}_f(x_0),\, d^{fs}(x,x_0)\le r,\,\,n\in\mathbb{N}\}<\infty.
$$

\end{prop}

\begin{proof}
%First, we note that by Lemma \ref{l:holder}, the map
%$\Phi: W^s_{f_0}(x)\to W^s_{f}(\Phi(x))$ is bi-H\"older.
%Since, by Proposition \ref{p:alg} the sequence of maps $\{f_0^{-n}g_0f_0^n|_{W^{fs}_{f_0}(x)}\}$ is uniformly H\"older,
%it follows that the sequence of maps $\{f^{-n}gf^n|_{W^{fs}_{f}(\Phi(x))}\}$ is uniformly H\"older as well.
Suppose that the claim fails, i.e., there exist sequences $x_i\in W^{fs}_f(x_0)$, $d^{fs}(x_i,x_0)\le r$, 
and $n_i\in\mathbb{N}$, $n_i\to\infty$, such that
$$
\|D(f^{-n_i}gf^{n_i})_{x_i}P_{x_i}\|\to\infty.
$$
Passing to a subsequence, we may assume that $x_i\to x_\infty$ for some $x_\infty\in W^{fs}_f(x_0)$ such that $d^{fs}(x_\infty,x_0)\le r$.
It follows from Proposition \ref{eq:main_est} that
$$
\|D(f^{-n_i}gf^{n_i})_{x_\infty}P_{x_\infty}\|\ge (1-c\cdot
d^{fs}(x_i,x_\infty)^\kappa)
\|D(f^{-n_i}gf^{n_i})_{x_i}P_{x_i}\|-\delta_{n_i}\to\infty.
$$
Let $v_i\in T_{x_\infty}(W^{fs}_{f}(x_\infty))$ with $\|v_i\|=1$ be such that
$$
\|D(f^{-n_i}gf^{n_i})_{x_\infty}P_{x_\infty}\|=\|D(f^{-n_i}gf^{n_i})_{x_\infty}v_i\|.
$$
Passing to a subsequence, we may assume that $v_i\to v_\infty$.
We have
\begin{align*}
\|D(f^{-n_i}gf^{n_i})_{x_\infty}v_\infty\|
&\ge \|D(f^{-n_i}gf^{n_i})_{x_\infty}v_i\|-
\|D(f^{-n_i}gf^{n_i})_{x_\infty}P_{x_\infty}(v_\infty-v_i)\|\\
&\ge \|D(f^{-n_i}gf^{n_i})_{x_\infty}P_{x_\infty}\|\cdot
(1-\|v_\infty-v_i\|).
\end{align*}
Hence, for sufficiently large $i$, we have
\begin{align*}
\|D(f^{-n_i}gf^{n_i})_{x_\infty}v_\infty\|\ge \frac{1}{2}\|D(f^{-n_i}gf^{n_i})_{x_\infty}P_{x_\infty}\|.
\end{align*}
%Hence, for a vector $v\in T_{x_\infty}(W^{fs}_{f}(x_\infty))$ such that $\|v-v_\infty\|<\frac{1}{3}$,
%\begin{align*}
%\|D(f^{-n_i}gf^{n_i})_{x_\infty}v\|&\ge
%\|D(f^{-n_i}gf^{n_i})_{x_\infty}v_\infty\|-\frac{1}{3}\|D(f^{-n_i}gf^{n_i})_{x_\infty}P_{x_\infty}\|\\
%&\ge  \frac{1}{6}\|D(f^{-n_i}gf^{n_i})_{x_\infty}P_{x_\infty}\|.
%\end{align*}
%We set $\alpha_n:=\|D(f^{-n}gf^{n})_{x_\infty}v_\infty\|\to\infty$.
Let $\alpha_n=\|D(f^{-n}gf^{n})_{x_\infty}v_\infty\|$. Note that $\alpha_{n_i}\to\infty$.

Fix small $\epsilon>0$. Let $x\in W^{fs}_f(x_\infty)$ be such that $d^{fs}(x,x_\infty)<\epsilon$ and
$v\in T_x W^{fs}_f(x)$ such that $\|v-v_\infty\|<\epsilon$. We have 
\begin{align*}
\|D(f^{-n}gf^n)_x v-D(f^{-n}gf^n)_{x_\infty} v_\infty\|\le& \|D(f^{-n}gf^n)_xP_x
v-D(f^{-n}gf^n)_{x_\infty}P_{x_\infty} v\|\\
&+ \|D(f^{-n}gf^n)_{x_\infty}P_{x_\infty} v-D(f^{-n}gf^n)_{x_\infty}P_{x_\infty} v_\infty\|\\
\ll& d^{fs}(x,x_\infty)^\kappa \|D(f^{-n}gf^n)_{x_\infty}P_{x_\infty}\|
+\delta_n\\
&+\|D(f^{-n}gf^n)_{x_\infty}P_{x_\infty}\|\cdot\|v-v_\infty\|\\
\ll& (\epsilon^\kappa\alpha_n+\delta_n)+\epsilon\alpha_n.
\end{align*}
For some $\rho=\rho(\epsilon)>0$, there exists a smooth curve $\ell:[0,\rho]\to W^{fs}_{f}(x_\infty)$
 such that 
\begin{align*}
&\ell(0)=x_\infty,\quad \ell'(0)=v_\infty,\quad \ell'(t)\in T_{\ell(t)}W^{fs}_{f}(\ell(t)),\\
&\hbox{diam}(\ell([0,\rho]))<\epsilon,\quad\quad \|\ell'(t)-\ell'(0)\|<\epsilon.
\end{align*}
We consider the sequence of curves $\ell_n=(f^{-n}gf^n)\ell$.
Note that $\|\ell_n'(0)\|=\alpha_n$, and it follows from the previous computation that,
choosing $\epsilon$ sufficiently small, 
$$
\|\ell_{n_i}'(t)-\ell_{n_i}'(0)\|\le \frac{1}{3} \|\ell_{n_i}'(0)\|
$$
for all $t\in [0,\rho]$ and sufficiently large $i$.
Since $\|\ell'_{n_i}(0)\|\to\infty$, it follows that
the distance between $\ell_{n_i}(0)$ and $\ell_{n_i}(\rho)$ in the ambient
Eucledean space goes to infinity as $i\to\infty$. This contradiction
proves the proposition.
\end{proof}

\begin{proof}[Proof of Theorem \ref{th:compact2}]
%It follows from Proposition \ref{p:alg} that
%the sequence of maps $\{f_0^{-n}g_0f_0^n|_{W^{fs}_{f_0}(x)}\}$ is uniformly H\"older.
%Since $f$ and $g$ are H\"older conjugate to $f_0$ and $g_0$,
%the maps $\{f^{-n}gf^n|_{W^{fs}_f(x)}\}$ are uniformly H\"older too.
By Proposition \ref{p:alg}(1), the maps $\{f^{-m_i}gf^{m_i}|_{W^{fs}_f(x)}\}$
are precompact in $C^0$-topology.
Then it follows from Proposition \ref{p:bound} that the maps $\{f^{-m_i}gf^{m_i}|_{W^{fs}_f(x)}\}$
are uniformly bounded in the $C^1$-topology. Also, combining Proposition \ref{eq:main_est} and
Proposition \ref{p:bound}, we obtain that for every $z$ and $w$ in a compact neighborhood of $x$ in $W^{fs}_f(x)$,
$$
\|D(f^{-m_i}gf^{m_i})_zP_z-D(f^{-m_i}gf^{m_i})_w P_w\|\ll d^{fs}(z,w)^\kappa+\delta_{m_i}.
$$
Since $\delta_{m_i}\to 0$, it follows that the maps
$\{f^{-m_i}gf^{m_i}|_{W^{fs}_f(x)}\}$  are equicontinuous in the $C^1$-topology. This implies the theorem.
\end{proof}

\subsection{H\"older implies $C^1$ along fast stable manifolds}\label{s4}

\begin{thm}\label{th:ci}
Let $f_0,g_0\in\Aff(X)$ be a good pair, $f,g\in \Diff(X)$ $C^1$-small perturbations of $f_0,g_0$,
and  $\Phi:X\to X$ a H\"older isomorphism such that
$$
\Phi\circ f_0=f\circ \Phi\quad\hbox{and}\quad \Phi\circ g_0=g\circ \Phi,
$$
and
$$
\Phi(W^{fs}_{f_0}(x))= W^{fs}_{f}(\Phi(x))\quad\hbox{for all $x\in X$.}
$$
Then for a.e. $x\in X$, the maps  $\Phi|_{W^{fs}_{f_0}(x)}$ and $\Phi^{-1}|_{W^{fs}_{f}(\Phi(x))}$
are $C^1$-diffeomorphisms.
\end{thm}

\begin{proof}
Fix a sequence $\{n_i\}$ and $A\in\hbox{GL}(W^{min}_{f_0})$ as in Proposition \ref{p:alg}(2).
For a set of $x\in X$ of full measure,
the sequence $\{f_0^{-n_i}g_0f_0^{n_i}(x)\}$  is dense in $X$.
In particular, for a.e. $x\in X$ and every $y\in W^{fs}_{f_0}(x)$, there exists a
subsequence $\{n_{i_j}\}$ such that $f_0^{-n_{i_j}}g_0f_0^{n_{i_j}}(x)\to y$.
Then by Proposition \ref{p:alg}(2), for every $v\in  W^{min}_{f_0}$,
\begin{equation}\label{eq:con}
f_0^{-n_{i_j}}g_0f_0^{n_{i_j}}(\exp(v)x)\to \exp (A v)y
\end{equation}
uniformly on compact sets.

For $k\in\mathbb{N}$ and $y\in W^{fs}_{f_0}(x)$, we consider maps
\begin{align*}
&\rho^0_{k,y}:W^{fs}_{f_0}(x)\to W^{fs}_{f_0}(x): \exp(v)x\mapsto
\exp(A^kv)y,\\
&\rho^1_{k,y}:W^{fs}_{f}(\Phi(x))\to W^{fs}_{f}(\Phi(x)): \Phi(\exp(v)x)\mapsto
\Phi\left(\exp(A^kv)y\right),
\end{align*} 
where $v\in W^{min}_{f_0}$. Note that
\begin{equation}\label{eq:conj0}
\rho^0_{k,y}=\Phi^{-1}\circ\rho^1_{k,y}\circ\Phi.
\end{equation}
In particular, it follows that $\rho^1_{k,y}$ is a homeomorphism, and
by (\ref{eq:con}),
$$
\rho^1_{1,y}=\lim_{j\to\infty} (f^{-n_{i_j}}gf^{n_{i_j}})|_{W^{fs}_f(\Phi(x))}.
$$
in the $C^0$-topology. By Theorem \ref{th:compact2}, 
the sequence of maps $(f^{-n_{i_j}}gf^{n_{i_j}})|_{W^{fs}_f(\Phi(x))}$ is precompact
in the $C^1$-topology. Hence, there exists a
subsequence which converges in the $C^1$-topology, and
$\rho^1_{1,y}$ is a $C^1$-map for every $y\in W^{fs}_{f_0}(x)$.
Since each map $\rho_{k,y}^1$, $k\ge 1$, is a composition of maps $\rho_{1,z}^1$, it is also $C^1$.

Next, we show that
\begin{equation}\label{eq:nonzero}
D(\rho^1_{1,y})_z\ne 0\quad\hbox{for every $y\in W^{fs}_{f_0}(x)$ and $z\in W^{fs}_f(\Phi(x))$.}
\end{equation}
Suppose that, to the contrary, $D(\rho^1_{1,y_0})_{z_0}= 0$ for some $y_0\in W^{fs}_{f_0}(x)$ and $z_0\in
W^{fs}_f(\Phi(x))$. 
We claim that for every $y\in W^{fs}_{f_0}(x)$, there exists $y_1\in W^{fs}_{f_0}(x)$ such
that
\begin{equation}\label{eq:rho}
\rho^0_{2,y}=\rho^0_{1,y_1}\rho^0_{1,y_0}.
\end{equation}
Indeed, if we write $y=\exp(v)x$, $y_1=\exp(v_1)x$, $y_0=\exp(v_0)x$
for some $v,v_1,v_0\in W^{min}_{f_0}$, then (\ref{eq:rho}) is equivalent to
$$
A^2w+v=A(Aw+v_0)+v_1,\quad w\in W^{min}_{f_0},
$$ 
and we can take $v_1=v-A v_0$.
Now by (\ref{eq:conj0}) and (\ref{eq:rho}),
$$
\rho^1_{2,y}=\rho^1_{1,y_1}\rho^1_{1,y_0}.
$$
Hence, $D(\rho^1_{2,y})_{z_0}= 0$ for every $y\in W^{fs}_{f_0}(x)$.
Similarly, using (\ref{eq:conj0}), we deduce that for every $z\in W^{fs}_f(\Phi(x))$,
there exists $y_z\in W^{fs}_{f_0}(x)$ such that $\rho^1_{1,y_z}(z)=z_0$.
If we fix $y_2\in W^{fs}_{f_0}(x)$, there exists $y_z'\in W^{fs}_{f_0}(x)$ such that
$$
\rho^1_{3,y_2}=\rho^1_{2,y_z'}\rho^1_{1,y_z}.
$$
Then we have
$$
D(\rho^1_{3,y_2})_z= 0\quad\hbox{for every $z\in W^{fs}_{f}(\Phi(x))$.}
$$
This contradicts the map $\rho^1_{3,y_2}$ being a homeomorphism, and (\ref{eq:nonzero}) follows.
We have proved that $\rho^1_{1,y}$ is a $C^1$-diffeomorphism for every $y\in W^{fs}_{f_0}(x)$.
This implies that the map $\rho^1_{0,y}$,
which can be represented as a composition of $\rho^1_{1,z_1}$ and $(\rho^1_{1,z_2})^{-1}$,
is also a $C^1$-diffeomorphism for every $y\in W^{fs}_{f_0}(x)$.

We have a free transitive action of $W^{min}_{f_0}$
on $W^{fs}_f(\Phi(x))$ defined by
\begin{equation}\label{eq:s}
s(v,\Phi(\exp(w)x))= \Phi(\exp(v+w)x)
\end{equation}
where $v,w\in W^{min}_{f_0}$. 
Note that the action $s:W^{min}_{f_0}\times W^{fs}_f(\Phi(x))\to W^{fs}_f(\Phi(x))$
is continuous, and since
$$
s(v,\Phi(\exp(w)x))=\rho^1_{0,\exp(v)x}(\Phi(\exp(w)x)),
$$
the map $s(v,\cdot)$ is a $C^1$-diffeomorphism for every $v\in W^{min}_{f_0}$.
Hence, by the Bochner--Montgomery theorem \cite{bm}, the map $s$ is $C^1$.
Now it follows from (\ref{eq:s}) that the map $\Phi_x(v):=\Phi(\exp(v)x)$ is $C^1$.

Suppose that for some $v\in W^{min}_{f_0}$, we have $\phi'(0)=0$ where $\phi(t)=\Phi(\exp(tv)x)$.
Then since $\phi(t_1+t)=\rho^1_{0,\exp(t v)x}(\Phi(\exp(t_1v)x))$, it follows that
$\phi'(t)=0$ for every $t\in \mathbb{R}$.
This contradicts the action $s$ being free. Hence, we conclude that
$D(\Phi_x)_0$ is nondegenerate, and because
$$
\Phi(\exp(v+w)x)=\rho^1_{0,\exp(v)x}(\Phi(\exp(w)x)),
$$
$D(\Phi_x)_v$ is nondegenerate  for every $v\in W^{min}_{f_0}$.
This shows that $\Phi|_{W^{fs}_{f_0}(x)}$ is a $C^1$-diffeomorphism for a.e. $x\in X$.
\end{proof}

\subsection{Completion of the proof of the main theorem }\label{s5}

Let $\{f_0,g_0\}\subset\alpha_0(\Gamma)$ be a good pair and $\{f,g\}\subset\alpha_1(\Gamma)$
its conjugate under $\Phi$. 
We use notation $A$, $\{n_i\}$, $\omega$ as in Proposition  \ref{p:alg} and Remark \ref{r:alg}.
Recall that 
$$
A=\lim_{i\to\infty} \omega^{-n_i} P^{min}_{f_0}(Dg_0)\omega^{n_i}.
$$
Hence, replacing the pair $\{f_0,g_0\}$ by the pair
$\{f_0,f_0^lg_0\}$ for some $l \ge 1$, we can get a good pair with $A$ satisfying $\|A\|<1$, which we now assume.

By Theorem \ref{th:ffs}, $\Phi(W^{fs}_{f_0}(x))=W^{fs}_{f}(\Phi(x))$ for all $x\in X$,
so we consider the maps
\begin{align*}
&\alpha^0_{x}:W^{fs}_{f_0}(x)\to W^{fs}_{f_0}(x): \exp(v)x\mapsto
\exp(Av)x,\\
&\alpha_{x}:W^{fs}_{f}(\Phi(x))\to W^{fs}_{f}(\Phi(x)): \Phi(\exp(v)x)\mapsto
\Phi\left(\exp(Av)x\right),
\end{align*}
where $v\in W^{min}_{f_0}$. Note that
\begin{equation}\label{eq:conj}
\Phi\circ \alpha^0_{x}=\alpha_{x}\circ\Phi.
\end{equation}
In particular, it follows that $\alpha_x$ is a homeomorphism.
For a.e. $x\in X$, there exists a subsequence $\{n_{i_j}\}$ such that
$f_0^{-n_{i_j}}g_0f_0^{n_{i_j}}(x)\to x$ as $j\to\infty$. Then by Proposition \ref{p:alg}, 
$$
\alpha^0_x=\lim_{j\to\infty} (f_0^{-n_{i_j}}g_0f_0^{n_{i_j}})|_{W^{fs}_{f_0}(x)},
$$
and by (\ref{eq:conj}),
$$
\alpha_x=\lim_{j\to\infty} (f^{-n_{i_j}}gf^{n_{i_j}})|_{W^{fs}_f(\Phi(x))}
$$
in the $C^0$-topology. It follows from Theorem \ref{th:compact2}
that the sequence of maps $(f^{-n_{i_j}}gf^{n_{i_j}})|_{W^{fs}_f(\Phi(x))}$
is precompact in the $C^1$-topology. Hence,
it contains a subsequence which converges in the $C^1$-topology, and
$\alpha_x$ is a $C^1$-map for a.e.  $x\in X$.

Let $W^{min}_{f_0,g_0}$ be the sum of eigenspaces of $A$ with eigenvalues of minimal modulus.
Our aim is to show that the map $\Phi$ restricted to the leaves $\exp(W^{min}_{f_0,g_0})x$ is linear
in  suitable $C^\infty$-coordinate systems which depend continuously on $x$. 
The first step is to show that the maps $\alpha_x$ are linear in suitable
coordinates on the fast stable leaves. Consider the measurable function
$$
\sigma(x)=\sup\{\|D(f^{-n}gf^n)_{\Phi(x)}P_{\Phi(x)}\|:\,\, n\in\mathbb{N}\},
$$
which is well defined by Proposition \ref{p:bound}. For $c>0$, let $X(c)$
be the subset of $x\in X$ such that $\sigma(x)\le c$ and the sequence
$\{f_0^{-n_i}g_0f_0^{n_i}(x)\}$ has $x$ as an accumulation point. By property (iv) of good pair and 
Proposition \ref{p:bound}, the set $\cup_{c>0} X(c)$ has full measure in $X$.

Let $Df_0|_{W^{min}_{f_0}}=\lambda\cdot \omega$ where $\lambda\in\mathbb{R}$, $|\lambda|<1$, and
$\omega\in\hbox{Isom}(W^{min}_{f_0})$. Using the Poincare recurrence theorem, 
for a.e. $(x,\omega')\in X(c)\times \hbox{Isom}(W^{min}_{f_0})$, one can construct
a sequence $\{k_j\}$, $k_0=0$, such that
\begin{align*}
f_0^{k_j}(x)\in X(c)\quad\hbox{for every $j\ge 1$}\quad\hbox{and}\quad
\omega^{k_j}\omega'\to \omega'\quad \hbox{as $j\to\infty$.}
\end{align*}
Then $\omega^{k_j}\to id$. Hence, by the Fubini theorem,
for a.e. $x\in X(c)$, there exists a sequence $\{k_j\}$, $k_0=0$, such that
\begin{equation}\label{eq:poincare}
f_0^{k_j}(x)\in X(c)\quad\hbox{for every $j\ge 1$}\quad\hbox{and}\quad
\omega^{k_j}\to id\quad \hbox{as $j\to\infty$.}
\end{equation}

Now we assume that $x\in X(c)$ satisfies (\ref{eq:poincare}).
Let $\{n^{(j)}_i\}$ be a subsequence such that
$$
(f_0^{-n^{(j)}_i}g_0f_0^{n^{(j)}_i})f_0^{k_j}(x)\to f_0^{k_j}(x)\quad\hbox{as $i\to\infty$,}
$$
Then by Proposition \ref{p:alg}(2),
\begin{equation}\label{eq:c0}
(f_0^{-n^{(j)}_i}g_0f_0^{n^{(j)}_i})|_{W^{fs}_{f_0}(f_0^{k_j}(x))}\to
\alpha^0_{f_0^{k_j}(x)}\quad\hbox{as $i\to\infty$}
\end{equation}
in the $C^0$-topology.
A direct computation shows that
$$
 \alpha_{x,j}^0= f_0^{-k_j}\circ\alpha^0_{f_0^{k_j}(x)}\circ f_0^{k_j}
$$
where
$$
\alpha^0_{x,j}:W^{fs}_{f_0}(x)\to W^{fs}_{f_0}(x): \exp(v)x\mapsto
\exp((\omega^{-k_j}A\omega^{k_j})v)x,\quad v\in W^{min}_f.
$$
Clearly, $\alpha^0_{x,j}\to \alpha^0_{x}$ as $j\to\infty$ in the $C^0$-topology.
It follows from (\ref{eq:conj}) that
\begin{equation}\label{eq:equiv}
 \alpha_{x,j}= f^{-k_j}\circ\alpha_{f_0^{k_j}(x)}\circ f^{k_j}
\end{equation}
where
$$
\alpha_{x,j}=\Phi\circ\alpha^0_{x,j}\circ\Phi^{-1} \to \alpha_x\quad\hbox{as $j\to\infty$}
$$
in the $C^0$-topology.

Since $f$ is $C^1$-close to the algebraic map $f_0$, its Mather spectrum on fast stable leaves is
contained in a small interval, and by the nonstationary Sternberg linearization \cite{GK,Guysinsky},
$f|_{W^{fs}_f(z)}$ is linear in suitable coordinate systems.
Namely, there exists a family of $C^\infty$-diffeomorphisms
$$
L_z:W^{min}_{f_0}\to W^{fs}_f(z),\quad z\in X,
$$
such that the map $z\mapsto L_z$ is continuous in the $C^\infty$-topology, $L_z(0)=z$, $D(L_z)_0=id$, and
\begin{equation}\label{eq:L}
(L_{f(z)}^{-1}\circ f\circ L_z)(v)=\rho(z)v, \quad v\in W^{min}_{f_0},
\end{equation}
with $\rho(z)\in\hbox{GL}(W^{min}_{f_0})$, $\|\rho(z)\|<1$.
Consider the sequence of maps
$$ 
G_k=L_{f^{k}(\Phi(x))}^{-1}\circ \alpha_{f_0^{k}(x)}\circ L_{f^{k}(\Phi(x))}:W^{min}_{f_0}\to W^{min}_{f_0}.
$$
We claim that the sequence of maps $G_{k_j}$ restricted to compact sets
is uniformly bounded and equicontinuous in the $C^1$-topology.
This is equivalent to the sequence $\{\alpha_{f_0^{k_j}(x)}\}$ being uniformly bounded
and equicontinuous in the $C^1$-topology.
It follows from (\ref{eq:conj}), (\ref{eq:c0}), and \eqref{eq:equiv} that
$$
F_{i,j}:=(f^{-k_j-n^{(0)}_i}gf^{n^{(0)}_i+k_j})|_{W^{fs}_{f}(\Phi(x))}\to
\alpha_{x,j}\quad\hbox{as $i\to\infty$}
$$
in the $C^0$-topology, and by Theorem \ref{th:compact2}, we may assume, after passing to a subsequence,
that convergence also holds in the $C^1$-topology. By (\ref{eq:equiv}),
\begin{align*}
\alpha_{f_0^{k_j}(x)}=&(f^{k_j}\circ\alpha_{x,j}\circ f^{-k_j})|_{W^{fs}_{f}(f^{k_j}(\Phi(x)))}.
\end{align*}
We observe that
$(f^{k_j}  \circ  F_{i,j} \circ f^{-k_j})|_{W^{fs}_{f}(f^{k_j}(\Phi(x)))}$ converges in the $C^1$-topology
to $\alpha _{f_0 ^{k_j} (x)}$ as $i\to\infty$, and since $f_0^{k_j}(x)\in X(c)$ for all $j$, the derivative of
$$(f^{k_j}  \circ  F_{i,j} \circ f^{-k_j})|_{W^{fs}_{f}(f^{k_j}(\Phi(x)))} = (f  ^{- n_i ^{(0)}} g   f^{ n_i ^{(0)}})|_{W^{fs}_{f}(f^{k_j}(\Phi(x)))}$$
is uniformly bounded over compact sets and $i\in\mathbb{N}$.  
This implies  that the sequence $\{\alpha_{f_0^{k_j}(x)}\}$
is uniformly bounded in the $C^1$-topology. To prove equicontinuity, we observe that
for $z,w\in W^{fs}_f(f^{k_j}(\Phi(x)))$, 
\begin{align*}
\left\|D(\alpha_{f_0^{k_j}(x)})_zP_z-D(\alpha_{f_0^{k_j}(x)})_w P_w\right\|
&\le\|D(f^{k_j}\circ(\alpha_{x,j}-F_{i,j})\circ f^{-k_j})_zP_z\|\\
&+\|D(f^{k_j}F_{i,j}f^{-k_j})_zP_z-D(f^{k_j}F_{i,j}f^{-k_j})_wP_w\|\\
&+\|D(f^{k_j}\circ(F_{i,j}-\alpha_{x,j})\circ f^{-k_j})_wP_w\|.
\end{align*}
Since $F_{i,j}\to \alpha_{x,j}$ as $i\to\infty$ in the $C^1$-topology, taking $i=i(j)$ sufficiently large,
we can make the first and the last terms arbitrary small.
To estimate the middle term, we use that $f_0^{k_j}(x)\in X(c)$ for all $j$ and Proposition \ref{eq:main_est}.
We get
$$
\|D(f^{k_j}F_{i,j}f^{-k_j})_zP_z-D(f^{k_j}F_{i,j}f^{-k_j})_wP_w\|\ll d^{fs}(z,w)^\kappa+\delta_{n_i^{(0)}},
$$
where $\delta_n\to 0$ as $n\to\infty$. This proves equicontinuity, and in fact, the stronger conclusion:
\begin{equation}\label{eq:eq}
\left\|D(\alpha_{f_0^{k_j}(x)})_zP_z-D(\alpha_{f_0^{k_j}(x)})_w P_w\right\|\ll d^{fs}(z,w)^\kappa.
\end{equation}
Let $\rho_k=\prod_{s=k-1}^{0} \rho(f^{s}(\Phi(x)))$ with $\rho$ defined as in (\ref{eq:L}). 
Since $f$ is $C^1$-close to the map $f_0$, which is conformal on the fast stable leaves,
it follows that for some $\lambda<1$ and small $\epsilon>0$, we have
\begin{equation}\label{eq:rho1}
\|\rho_k(x)\|\ll (\lambda+\epsilon)^k\quad\hbox{and}\quad
\|\rho_k(x)^{-1}\|\ll (\lambda-\epsilon)^{-k}
\end{equation}
uniformly on $x\in X$ and $k\in \mathbb{N}$.
We deduce from (\ref{eq:equiv}) and (\ref{eq:L}) that
\begin{equation}\label{eq:g_0}
G_{0,j}=\rho_{k_j}^{-1} G_{k_j}(\rho_{k_j} v),\quad v\in W^{min}_{f_0},
\end{equation}
where $G_{0,j}=L_x^{-1}\circ\alpha_{x,j}\circ L_x\to G_0$ as $j\to\infty$.
Fix a basis $\{e_\ell\}$ of $W^{min}_{f_0}$ and write
$$
G_k(v)=\sum_\ell G_{k,\ell}(v)e_\ell.
$$
Applying the mean value theorem to the functions $t\mapsto G_{k_j,\ell}(t\rho_{k_j}v)$, $t\in [0,1]$,
we deduce that
\begin{equation}\label{eq:bj}
G_{k_j,\ell}(\rho_{k_j}v)=\sum_s \frac{\partial G_{k_j,\ell}}{\partial x_s}(t_{j,\ell}(v)\rho_{k_j}v)(\rho_{k_j}v)_s
\end{equation}
for some $t_{j,\ell}(v)\in [0,1]$.
Hence, by (\ref{eq:g_0}),
\begin{equation}\label{eq:G}
G_{0,j}(v)=(\rho_{k_j}^{-1}B_j(v)\rho_{k_j}) v
\end{equation}
where $B_j(v)$ is the linear map of $W^{min}_{f_0}$ with coefficients coming from \eqref{eq:bj}.
Let $B_j=D(G_{k_j})_{0}$. From \eqref{eq:eq}, we deduce that
$$
\|B_j(v)-B_j\|\ll \|\rho_{k_j}v\|^\kappa,
$$
and it follows from \eqref{eq:rho1} that
\begin{equation}\label{eq:est}
\|\rho_{k_j}^{-1}B_j(v)\rho_{k_j}-\rho_{k_j}^{-1}B_j\rho_{k_j}\|\to 0\quad\hbox{as $j\to\infty$.}
\end{equation}
 Let $B=B(x)$ be a limit point of the sequence of maps
$\rho_{k_j}^{-1}B_j(v)\rho_{k_j}$. The crucial point of our argument is that $B(x)$ is independent
of $v$ because of the equicontinuity estimate.
Taking $j\to\infty$, we deduce from (\ref{eq:G}) that $G_0(v)=B\, v$, and by the definition of $G_0$,
\begin{equation}\label{eq:Bx}
L_x^{-1}(\Phi(\exp(Av)x))= B(x)\, L_x^{-1}(\Phi(\exp(v)x))).
\end{equation}
This equality holds for a.e. $x\in X(c)$ with $c>0$ and hence, for a.e. $x\in X$.
Note that since $\Phi$ is a homeomorphism, the linear
map $B(x)$ is nondegenerate. Although $B(x)$ is defined only for a.e. $x\in X$, it
follows from \eqref{eq:Bx} that it can be extended continuously to the whole space
so that \eqref{eq:Bx} holds everywhere.

Now we consider the maps
$$
\Phi_x(v)=L_{\Phi(x)}^{-1}(\Phi(\exp(v))x),\quad x\in X,\; v\in W^{min}_{f_0},
$$
which satisfy the equivariance relation $\Phi_x\circ A=B(x)\circ \Phi_x$.
Recall that by Theorem \ref{th:ci}, $\Phi_x$ is a $C^1$-diffeomorphism for a.e. $x\in X$.
Hence, we have $D(\Phi_x)_0A=B(x)D(\Phi_x)_0$, and it follows that the map
$\Psi_x:=D(\Phi_x)_0^{-1}\circ \Phi_x$ commutes with the contraction $A$.
We write $A|_{W^{min}_{f_0,g_0}}=\rho\cdot \theta$ where $\rho\in\mathbb{R}$, $|\rho|<1$,
 and $\theta\in \hbox{Isom}(W^{min}_{f_0,g_0})$.
Since 
$$
W^{min}_{f_0,g_0}=\{v\in W^{min}_{f_0}:\, \|A^nv\|=O(\rho^n)\quad\hbox{as $n\to\infty$}\},
$$
and $\Psi_x$ is a $C^1$-map, we deduce that 
$$
\Psi_x(W^{min}_{f_0,g_0})\subset W^{min}_{f_0,g_0}.
$$
We claim that $\Psi_x|_{W^{min}_{f_0,g_0}}$ is linear.
We fix a basis $\{e_\ell\}$ of $W^{min}_{f_0,g_0}$ and write $\Psi:=\Psi_x$ as
$$
\Psi(v)=\sum_\ell \Psi_\ell(v)e_\ell,\quad v\in W^{min}_{f_0,g_0}.
$$
By the mean value theorem,
$$
\Psi_\ell(A^n v)=\sum_s \frac{\partial\Psi_\ell}{\partial x_s}(t_\ell(v)A^n v)(A^n v)_s
$$
for some $t_\ell(v)\in [0,1]$. Hence,
\begin{equation}\label{eq:psi}
\Psi(v)=A^{-n}\Psi(A^n v)=(A^{-n}C_n(v)A^n) v=(\theta^{-n}C_n(v)\theta^n) v
\end{equation}
where $C_n(v)$ is the matrix with coefficients $\frac{\partial\Psi_\ell}{\partial x_s}(t_\ell(v)A^n v)$.
Since $\Psi$ is a $C^1$-map and $\|A\|<1$, it follows that $C_n(v)\to D(\Psi)_0$ as $n\to\infty$. Passing to a subsequence,
we may assume that the sequence of isometries $\theta^n$ also converges. Then it follows from
(\ref{eq:psi}) that $\Psi$ is linear.
We conclude that for a.e. $x\in X$, there exists a linear map $C(x):W^{min}_{f_0,g_0}\to W^{min}_{f_0,g_0}$
such that 
\begin{equation}\label{eq:last}
\Phi(\exp(v)x)=L_{\Phi(x)}(C(x)v),\quad v\in W^{min}_{f_0,g_0}.
\end{equation}
It follows from this relation that $C(x)$ is nondegenerate. Moreover,
by continuity, we may assume that (\ref{eq:last}) holds for all $x\in X$,
and $C(x)$ depends continuously on $x$. Now relation (\ref{eq:last}) also implies that
$\Phi$ is a $C^\infty$-diffeomorphism along the leaves $\exp(W^{min}_{f_0,g_0})x$, and the
partial derivatives along this leaves depend continuously on $x\in X$.

Note that if $\{f_0,g_0\}$ is a good pair, then $\{h^{-1}f_0h,h^{-1}f_0h\}$ is good as well for every
$h\in \alpha_0(\Gamma)$, and we have $W^{min}_{h^{-1}f_0h,h^{-1}g_0h}=(Dh)^{-1}W^{min}_{f_0,g_0}$.
Hence, it follows from the irreducibility of the $\Gamma$-action on $\hbox{Lie}(G)$ that
\begin{equation}\label{eq:irrr}
\sum_{\{f_0,g_0\}\subset\alpha_0(\Gamma)\hbox{\tiny -good}} W^{min}_{f_0,g_0}=\hbox{\rm Lie}(G).
\end{equation}
Now we consider the elliptic differential operator $\mathcal{D}^s=\sum_i \frac{\partial^{2s}}{\partial
  x_i^{2s}}$ where the partial derivatives $\frac{\partial}{\partial x_i}$ span the tangent space
and are taken in directions of $W^{min}_{f_0,g_0}$
for some choice of good pairs $\{f_0,g_0\}$. It follows from the previous paragraph
that $\mathcal{D}^s\Psi$ is continuous for every $s\ge 2$. Hence, by the regularity of
solutions of elliptic PDE, $\Psi$ is $C^\infty$. 
Since $D(\Phi)_x$ is onto when restricted to fast stable distributions of good $f_0$ and its conjugate $f$,
it follows that $D(\Phi)_x$ is onto as well. Thus, $\Psi^{-1}$ is $C^\infty$ by the inverse function theorem.

\section{Existence of good pairs}\label{sec:gen}

\subsection{Tori}

In this section, we set $X=\mathbb{T}^d$, $d\ge 2$, and prove

\begin{prop}\label{p:torus} 
Let $\Gamma$ be a subgroup  of $\Aff(X)$ such that the Zariski
closure of $D\Gamma$ contains $\SL_d$. Then $\Gamma$ contains a good pair.
\end{prop}

We will use the following lemma, which is easy to prove using Fourier analysis (see, for example,
\cite[Corollary~1.6 and Remark~1.8]{bg}). Let $\phi$ be the Euler totient function.

\begin{lemma}\label{th:bg}
Let $f_1,f_2\in \Aff(X)$ be such that for every $l\ge 1$ satisfying $\phi(l)\le d^2$,
the map $Df_1^{-l}Df_2^l$ does not have eigenvalue 1.
Then for every $\phi_1,\phi_2\in L^2(X)$,
\begin{equation*}\label{eq:mixing}
\int_X \phi_1(f_1^n(x))\phi_2(f_2^n(x))d\mu(x)\to \left(\int_X
  \phi_1\,d\mu\right) \left(\int_X \phi_2\,d\mu\right)\quad\hbox{as $n\to\infty$.}
\end{equation*}
\end{lemma}
If %(\ref{eq:mixing})
the conclusion of Lemma \ref{th:bg} holds, then we call the pair $\{f_1,f_2\}$ {\it mixing}.
Mixing pairs can be used to construct affine maps satisfying property (iv) of good pairs.

\begin{lemma}\label{lem:dense}
Let $f,g\in \Aff(X)$ and suppose the pair $\{f^{-1}, gf^{-1}g^{-1}\}$ is mixing.
Then for every subsequence $\{n_i\}$ and for a.e. $x\in X$, the sequence
$\{f^{-n_i}gf^{n_i}(x)\}_{n\ge 0}$ is dense in $X$.
\end{lemma}

\begin{proof}
We have
$$
\int_X \phi_1(g f^{-n}g^{-1}(x))\phi_2(f^{-n}(x))d\mu(x)\to \left(\int_X
  \phi_1\,d\mu\right) \left(\int_X \phi_2\,d\mu\right)\quad\hbox{as $n\to\infty$}
$$
for every $\phi_1,\phi_2\in L^2(X)$. By invariance of the measure, this also
implies that
$$
\int_X \phi_1(x)\phi_2(f^{-n}gf^n(x))d\mu(x)\to \left(\int_X
  \phi_1\,d\mu\right) \left(\int_X \phi_2\,d\mu\right)\quad\hbox{as $n\to\infty$}
$$
for every $\phi_1,\phi_2\in L^2(X)$.

Now we show that for $\delta_n=f^{-n}gf^n$, the
sequence $\{\delta_{n_i}x\}$ is dense in $X$ for a.e. $x\in X$.
Let $U$ be a nonempty open subset of $X$ and $A=\cup_{i\ge 0}
\delta_{n_i}^{-1}(U)$. We have
$$
0=\int_X \chi_U(\delta_{n_i} (x))\chi_{A^c}(x)\,d\mu(x)\to \mu(U)\mu(A^c).
$$
This implies that $\mu(A^c)=0$, i.e. for a.e. $x\in X$,
$$
\{\delta_{n_i}x\}_{i\ge 0}\cap U\ne \emptyset.
$$
Since $X$ has countable base of topology, this proves the lemma.
\end{proof}

\begin{proof}[Proof of Proposition \ref{p:torus}]
Since $D\Gamma$ is Zariski dense, there is $f\in\Gamma$ such that $Df$ is $\mathbb{R}$-regular
(see \cite{bl,p}).
In particular, $Df$ is semisimple and hyperbolic.
Because of Lemmas \ref{th:bg} and \ref{lem:dense}, it suffices to
find $g\in\Gamma$ such that $Dg$ belongs to the set
\begin{align*}
\left\{X\in\hbox{SL}_d:\;\; \det(P^{min}_fX|_{W^{min}_f})\ne 0,\quad \det([Df^l,X]-id)\ne 0\;\;\hbox{for $\phi(l)\le d^2$}\right\}.
\end{align*}
One can check that this is a nonempty Zariski open subset of $\hbox{SL}_d$.
Hence, existence of such $g\in\Gamma$ follows from Zariski density.
\end{proof}

\subsection{Semisimple groups}
Let $G$ be a connected semisimple Lie groups with no compact factors,
$\Lambda$ a lattice in $G$, and $X=G/\Lambda$.

\begin{prop}\label{p:semi}
Let $\Gamma$ be a subgroup of $\Aff(X)$ such that
the Zariski closure of $D\Gamma$ contains $\hbox{\rm Ad}(G)$. Then $\Gamma$ contains a good pair.
\end{prop}

\begin{proof}
Since $D\Gamma$ contains a finite index subgroup consisting of inner automorphisms,
we may assume without loss of generality that $D\Gamma$ is a subgroup of $\hbox{Ad}(G)$.
It follows from Zariski density \cite{bl,p} that $\Gamma$ contains an element $f$
such that $Df$ is $\mathbb{R}$-regular.
In particular, it is partially hyperbolic and semisimple, and hence it satisfies properties
(i)--(ii) of the definition of a good pair.
If we choose $g\in\Gamma$ so that the pair $\{f^{-1},gf^{-1}g^{-1}\}$ is mixing, then
by Lemma \ref{lem:dense}, $f$ and $g$ will satisfy property (iv) of the definition of
a good pair. By the Howe--Moore theorem, the pair $\{f^{-1},gf^{-1}g^{-1}\}$
is mixing provided that for all projections
$\pi_i:\hbox{Ad}(G)\to \hbox{Ad}(G_i)$ on simple factors of $\hbox{Ad}(G)$, the sequence
$\{\pi_i(Dg(Df)^{-n}(Dg)^{-1}(Df)^n)\}$ is divergent. Since $\pi_i(Df)$ is also
$\mathbb{R}$-regular,
$$
P_i=\{g\in G_i:\, \pi_i(Df)^{-n}\cdot g\cdot \pi_i(Df)^{n}\hbox{ is nondivergent}\}
$$
is a proper parabolic subgroup of $G_i$. By Zariski density, there exists $g\in \Gamma$
such that $\pi_i(Dg)\notin P_i$ for all $i$, and  $P_f^{min}(Dg):W^{min}_f\to W^{min}_f$ is nondegenerate.
Such $f$ and $g$ provide a good pair.
\end{proof}

\ignore{
A simple modification of the above argument also gives

\begin{prop}\label{p:semi2}
Assume that $G$ is simple. Let $\Gamma$ be a subgroup of $\Aff(X)$ such that
the Zariski closure of $D\Gamma$ is noncompact and semisimple. Then $\Gamma$ contains a good pair.
\end{prop}
}

\end{document}